\documentclass[oneside,reqno,10pt]{amsart}

\usepackage[margin=1in]{geometry}
\usepackage{mathtools}
\usepackage{amsmath}
\usepackage{amsthm}
\usepackage{amssymb}
\usepackage{tikz}
\usepackage{tikz-cd}
\usepackage[all,cmtip]{xy}
\usepackage{float}
\usepackage{indentfirst}
\usepackage[mathscr]{euscript}
\usepackage{stmaryrd}
\usepackage{hyperref}
\usepackage{csquotes}

\makeatletter
\def\blfootnote{\gdef\@thefnmark{}\@footnotetext}
\makeatother

\theoremstyle{plain}
\newtheorem{thm}{Theorem}[section]
\newtheorem{prop}[thm]{Proposition}
\newtheorem{lem}[thm]{Lemma}
\newtheorem{cor}[thm]{Corollary}
\newtheorem{que}[thm]{Question}

\theoremstyle{definition}
\newtheorem{dfn}[thm]{Definition}
\newtheorem{exa}[thm]{Example}
\newtheorem*{ack}{Acknowledgements}

\theoremstyle{remark}
\newtheorem{rmk}[thm]{Remark}

\numberwithin{equation}{section}

\title{Categorical entropy, (co-)t-structures and ST-triples}
\author{Jongmyeong Kim}
\address{Center for Geometry and Physics, Institute for Basic Science (IBS), Pohang 37673, Republic of Korea}
\email{myeong@ibs.re.kr}

\begin{document}

\begin{abstract}
In this paper, we study a dynamical property of an exact endofunctor $\Phi : \mathcal{D} \to \mathcal{D}$ of a triangulated category $\mathcal{D}$.
In particular, we are interested in the following question: Given full triangulated subcategories $\mathcal{A},\mathcal{B} \subset 
\mathcal{D}$ such that $\Phi(\mathcal{A}) \subset \mathcal{A}$ and $\Phi(\mathcal{B}) \subset \mathcal{B}$, how the categorical entropies of $\Phi|_\mathcal{A}$ and $\Phi|_\mathcal{B}$ are related?
To answer this question, we introduce new entropy-type invariants using bounded (co-)t-structures with finite (co-)hearts and prove their basic properties.
We then apply these results to answer our question for the situation where $\mathcal{A}$ has a bounded t-structure and $\mathcal{B}$ has a bounded co-t-structure which are, in some sense, dual to each other.
\end{abstract}

\maketitle


\section{Introduction}

\subsection{Categorical dynamical systems}

A {\em categorical dynamical system} $(\mathcal{D},\Phi)$ is a pair of a triangulated category $\mathcal{D}$ and an exact endofunctor $\Phi : \mathcal{D} \to \mathcal{D}$.
The study of a categorical dynamical system as an analogue of a topological dynamical system was initiated by Dimitrov--Haiden--Katzarkov--Kontsevich \cite{DHKK}.
As in the case of a topological dynamical system, they introduced the notion of {\em categorical entropy} $h_t(\Phi)$ as an invariant measuring the dynamical complexity of a categorical dynamical system $(\mathcal{D},\Phi)$ (see Definition \ref{ent}).

The categorical entropy shares many properties with the topological entropy.
However there are still many elementary properties of the topological entropy which do not have categorical counterparts.
For instance, suppose we have a topological dynamical system $(X,f)$, i.e., a pair of a Hausdorff space $X$ and a continuous self-map $f : X \to X$, and subspaces $Y_1,\dots,Y_n \subset X$ such that $X = \cup_{i=1}^n Y_i$ and $f(Y_i) \subset Y_i$ for every $1 \leq i \leq n$.
Then we have topological dynamical systems $(Y_1,f|_{Y_1}),\dots,(Y_n,f|_{Y_n})$ and $(X,f)$.
It is known that their topological entropies satisfy
\begin{equation*}
h_\mathrm{top}(f) = \max_{1 \leq i \leq n} h_\mathrm{top}(f|_{Y_i})
\end{equation*}
(see \cite[Section 1.6]{Gro}).
For the categorical entropy, one can consider an analogous question.

\begin{que}
Let $(\mathcal{D},\Phi)$ be a categorical dynamical system and $\mathcal{C}_1,\dots,\mathcal{C}_n \subset \mathcal{D}$ be full triangulated subcategories such that the smallest thick subcategory of $\mathcal{D}$ containing $\mathcal{C}_1,\dots,\mathcal{C}_n$ coincides with $\mathcal{D}$ and $\Phi(\mathcal{C}_i) \subset \mathcal{C}_i$ for every $1 \leq i \leq n$.
Then does it hold that
\begin{equation*}
h_t(\Phi) = \max_{1 \leq i \leq n} h_t(\Phi|_{\mathcal{C}_i})?
\end{equation*}
\end{que}

If $\mathcal{D}$ admits a (weak) semiorthogonal decomposition $\langle\mathcal{C}_1,\dots,\mathcal{C}_n\rangle$ and $\Phi : \mathcal{D} \to \mathcal{D}$ is an exact autoequivalence such that $\Phi(\mathcal{C}_i) \subset \mathcal{C}_i$ for every $1 \leq i \leq n$ then this is true (see Corollary \ref{sod}).
However, to the author's knowledge, this is still unknown for general $\mathcal{C}_1,\dots,\mathcal{C}_n$.

In this paper, we consider a similar question.

\begin{que}\label{ques}
Let $(\mathcal{D},\Phi)$ be a categorical dynamical system and $\mathcal{A},\mathcal{B} \subset \mathcal{D}$ be full triangulated subcategories such that $\Phi(\mathcal{A}) \subset \mathcal{A}$ and $\Phi(\mathcal{B}) \subset \mathcal{B}$.
Then what is the relationship between $h_t(\Phi|_\mathcal{A})$ and $h_t(\Phi|_\mathcal{B})$?
\end{que}

\subsection{Main theorems}

To give a partial answer to Question \ref{ques}, we will make use of (co-)t-structures.
First, for a given categorical dynamical system $(\mathcal{D},\Phi)$ and a bounded t-structure $\mathfrak{a}$ with finite heart, we define a new entropy-type invariant $\hat{h}_t^\mathfrak{a}(\Phi)$ (see Definition \ref{t-ent}).
Similarly we also define another entropy-type invariant $\check{h}_t^\mathfrak{b}(\Phi)$ using a bounded co-t-structure $\mathfrak{b}$ with finite heart (see Definition \ref{c-ent}).
These invariants can be considered as analogues of the mass growth introduced by Dimitrov--Haiden--Katzarkov--Kontsevich \cite{DHKK} and studied by Ikeda \cite{Ike}.

Our first main theorem states that $\hat{h}_t^\mathfrak{a}(\Phi)$ and $\check{h}_t^\mathfrak{b}(\Phi)$ coincide with the categorical entropy $h_t(\Phi)$.

\begin{thm}[see Theorems \ref{t-thm-1} and \ref{c-thm-1}]\label{main-1}
Let $\mathcal{D}$ be a $\mathbf{K}$-linear idempotent complete triangulated category and $\Phi : \mathcal{D} \to \mathcal{D}$ be an exact endofunctor.
\begin{enumerate}
\item Let $\mathfrak{a}$ be a bounded t-structure on $\mathcal{D}$ with finite heart.
Then we have
\begin{equation*}
h_t(\Phi) = \hat{h}_t^\mathfrak{a}(\Phi).
\end{equation*}
\item Let $\mathfrak{b}$ be a bounded co-t-structure on $\mathcal{D}$ with finite co-heart.
Then we have
\begin{equation*}
h_t(\Phi) = \check{h}_t^\mathfrak{b}(\Phi).
\end{equation*}
\end{enumerate}
\end{thm}

One of the advantages of the new invariants $\hat{h}_t^\mathfrak{a}(\Phi)$ and $\check{h}_t^\mathfrak{b}(\Phi)$ is that they are easier to compute than the categorical entropy.
More importantly, they allow us to compute the categorical entropy via a linear algebraic computation as our second main theorem states.
It states that if there is an exact functor $\Gamma : \mathcal{D} \to D^b(\mathbf{K})$ satisfying a certain condition then $\hat{h}_t^\mathfrak{a}(\Phi)$ and $\check{h}_t^\mathfrak{b}(\Phi)$ (and hence also $h_t(\Phi)$ by Theorem \ref{main-1}) can be computed as the exponential growth rate of the dimensions of a sequence of some vector spaces.

\begin{thm}[see Theorems \ref{t-thm-2} and \ref{c-thm-2}]\label{main-2}
Let $\mathcal{D}$ be a $\mathbf{K}$-linear idempotent complete triangulated category, $\Phi : \mathcal{D} \to \mathcal{D}$ be an exact endofunctor and $\Gamma : \mathcal{D} \to D^b(\mathbf{K})$ (resp. $\Gamma : \mathcal{D}^\mathrm{op} \to D^b(\mathbf{K})$) be an exact functor.
\begin{enumerate}
\item Let $\mathfrak{a}$ be a bounded t-structure on $\mathcal{D}$ with finite heart.
Assume that $\Gamma(E) \cong \mathbf{K}$ for every simple object $E \in \mathcal{H}_\mathfrak{a}$.
Then we have
\begin{gather*}
\hat{h}_t^\mathfrak{a}(\Phi) = \lim_{N \to \infty} \frac{1}{N} \log \sum_{k \in \mathbf{Z}} \dim H^{-k}(\Gamma(\Phi^N(G))) e^{kt}\\
\left( \text{resp. } \hat{h}_t^\mathfrak{a}(\Phi) = \lim_{N \to \infty} \frac{1}{N} \log \sum_{k \in \mathbf{Z}} \dim H^{-k}(\Gamma(\Phi^N(G))) e^{-kt} \right)
\end{gather*}
for any split-generator $G$ of $\mathcal{D}$.
\item Let $\mathfrak{b}$ be a bounded co-t-structure on $\mathcal{D}$ with finite co-heart.
Assume that $\Gamma(E) \cong \mathbf{K}$ for every indecomposable object $E \in \mathcal{C}_\mathfrak{b}$ and, for any indecomposable objects $E,F \in \mathcal{C}_\mathfrak{b}$, a morphism $f \in \mathrm{Hom}(E,F)$ is an isomorphism whenever $\Gamma(f) \neq 0$.
Then we have
\begin{gather*}
\check{h}_t^\mathfrak{b}(\Phi) = \lim_{N \to \infty} \frac{1}{N} \log \sum_{k \in \mathbf{Z}} \dim H^{-k}(\Gamma(\Phi^N(G))) e^{kt}\\
\left( \text{resp. } \check{h}_t^\mathfrak{b}(\Phi) = \lim_{N \to \infty} \frac{1}{N} \log \sum_{k \in \mathbf{Z}} \dim H^{-k}(\Gamma(\Phi^N(G))) e^{-kt} \right)
\end{gather*}
for any split-generator $G$ of $\mathcal{D}$.
\end{enumerate}
\end{thm}

\subsection{Applications}

Let $\mathcal{D}$ be a $\mathbf{K}$-linear idempotent complete triangulated category.
Using Theorems \ref{main-1} and \ref{main-2}, we give an answer to Question \ref{ques} for an {\em ST-triple} $(\mathcal{B},\mathcal{A},M)$ inside $\mathcal{D}$ whose notion was introduced by Adachi--Mizuno--Yang \cite{AMY}.
Roughly speaking, it is a triple of a thick subcategory $\mathcal{A} \subset \mathcal{D}$ admitting a bounded t-structure $\mathfrak{a}$ with finite heart, a thick subcategory $\mathcal{B} \subset \mathcal{D}$ admitting a bounded co-t-structure $\mathfrak{b}$ with finite co-heart and a silting object $M \in \mathcal{B}$ which are closely related to each other (see Definition \ref{st-triple} and Example \ref{st-ex}).
In particular, if we have a categorical dynamical system $(\mathcal{D},\Phi)$ such that $\Phi(\mathcal{A}) \subset \mathcal{A}$ and $\Phi(\mathcal{B}) \subset \mathcal{B}$, we can apply Theorems \ref{main-1} (1) and \ref{main-2} (1) to $(\mathcal{A},\Phi|_\mathcal{A})$ and Theorems \ref{main-1} (2) and \ref{main-2} (2) to $(\mathcal{B},\Phi|_\mathcal{B})$.

Now consider an ST-triple $(\mathcal{B},\mathcal{A},M)$ inside $\mathcal{D}$.
Let $S_1,\dots,S_n$ be a complete set of pairwise non-isomorphic simple objects of the heart of the t-structure $\mathfrak{a}$ on $\mathcal{A}$ and $S \coloneqq S_1 \oplus \cdots \oplus S_n$.
Note that $S$ is a split-generator of $\mathcal{A}$ and $M$ is a split-generator of $\mathcal{B}$.
Applying Theorems \ref{main-1} (1) and \ref{main-2} (1) to $\Phi|_\mathcal{A}$ and $\Gamma_\mathcal{A} \coloneqq \mathrm{RHom}(M,-) : \mathcal{A} \to D^b(\mathbf{K})$ and Theorems \ref{main-1} (2) and \ref{main-2} (2) to $\Phi|_\mathcal{B}$ and $\Gamma_\mathcal{B} \coloneqq \mathrm{RHom}(-,S) : \mathcal{B} \to D^b(\mathbf{K})^\mathrm{op}$, we obtain the following.

\begin{thm}[see Theorem \ref{app-thm}]
Let $\mathcal{D}$ be a $\mathbf{K}$-linear idempotent complete algebraic triangulated category, $\Phi : \mathcal{D} \to \mathcal{D}$ be an exact endofunctor and $(\mathcal{B},\mathcal{A},M)$ be an ST-triple inside $\mathcal{D}$ such that $M$ is basic.
Assume that $\Phi(\mathcal{A}) \subset \mathcal{A}$ and $\Phi(\mathcal{B}) \subset \mathcal{B}$.
Then we have
\begin{gather*}
h_t(\Phi|_\mathcal{A}) = \lim_{N \to \infty} \frac{1}{N} \log \sum_{k \in \mathbf{Z}} \dim \mathrm{Hom}^{-k}(M,\Phi^N(G_\mathcal{A})) e^{kt},\\
h_t(\Phi|_\mathcal{B}) = \lim_{N \to \infty} \frac{1}{N} \log \sum_{k \in \mathbf{Z}} \dim \mathrm{Hom}^{-k}(\Phi^N(G_\mathcal{B}),S) e^{-kt}
\end{gather*}
for any split-generators $G_\mathcal{A} \in \mathcal{A}$ and $G_\mathcal{B} \in \mathcal{B}$.
In particular, if $\Phi$ is an exact autoequivalence then
\begin{equation*}
h_t(\Phi|_\mathcal{A}) = h_{-t}(\Phi|_\mathcal{B}^{-1}) = \lim_{N \to \infty} \frac{1}{N} \log \sum_{k \in \mathbf{Z}} \dim \mathrm{Hom}^{-k}(M,\Phi^N(S)) e^{kt}.
\end{equation*}
\end{thm}

\begin{rmk}
We expect that $h_{-t}^\mathcal{B}(\Phi|_\mathcal{B}^{-1}) = h_t^\mathcal{B}(\Phi|_\mathcal{B})$ holds maybe under some mild assumption.
Indeed this is the case if $\mathcal{B}$ is a saturated triangulated category (see \cite[Lemma 2.11]{FFO}).
\end{rmk}

\begin{ack}
This work was supported by the Institute for Basic Science (IBS-R003-D1).
\end{ack}

\section{Preliminaries}

\subsection{Notations and terminologies}

In this paper, $\mathbf{K}$ always denotes an algebraically closed field.
In particular, the algebraically closedness will be used in the proof of Theorem \ref{app-thm} (see Remark \ref{alg-cl}).

Let $\mathcal{D}$ be a $\mathbf{K}$-linear triangulated category.
The shift functor of $\mathcal{D}$ will be denoted by $\Sigma_\mathcal{D}$ or simply by $\Sigma$ if $\mathcal{D}$ is clear from the context.
The opposite category of $\mathcal{D}$ will be denoted by $\mathcal{D}^\mathrm{op}$.
The opposite category $\mathcal{D}^\mathrm{op}$ has a natural structure of a triangulated category with the shift functor $\Sigma_{\mathcal{D}^\mathrm{op}} = \Sigma_\mathcal{D}^{-1}$.
For an exact endofunctor $\Phi : \mathcal{D} \to \mathcal{D}$, its opposite functor will be denoted by $\Phi^\mathrm{op} : \mathcal{D}^\mathrm{op} \to \mathcal{D}^\mathrm{op}$.

By the notation $E \in \mathcal{D}$, we mean that $E$ is an object of $\mathcal{D}$.
For $E,F \in \mathcal{D}$, the $\mathbf{K}$-vector space of morphisms from $E$ to $F$ will be denoted by $\mathrm{Hom}_\mathcal{D}(E,F)$ or simply by $\mathrm{Hom}(E,F)$ if $\mathcal{D}$ is clear from the context.
We sometimes use the notation $\mathrm{Hom}_\mathcal{D}^k(E,F) \coloneqq \mathrm{Hom}_\mathcal{D}(E,\Sigma^kF)$.
We say $\mathcal{D}$ is {\em Hom-finite} if $\mathrm{Hom}_\mathcal{D}(E,F)$ is finite-dimensional for every $E,F \in \mathcal{D}$.

An exact triangle $D \to E \to F \to \Sigma D$ in $\mathcal{D}$ will often be depicted as
\begin{equation*}
\begin{tikzcd}[column sep=tiny]
D \ar[rr] & & E \ar[dl]\\
& F \ar[ul,dashed] &
\end{tikzcd}.
\end{equation*}
For $E \in \mathcal{D}$, we call a diagram in $\mathcal{D}$ of the form
\begin{equation*}
\begin{tikzcd}[column sep=tiny]
0 \ar[rr] & & * \ar[dl] \ar[rr] & & * \ar[dl] & \cdots & * \ar[rr] & & E \ar[dl]\\
& E_1 \ar[ul,dashed] & & E_2 \ar[ul,dashed] & & & & E_m \ar[ul,dashed] &
\end{tikzcd}
\end{equation*}
a {\em filtration} of $E$.
Here $*$ in the diagram denotes some object of $\mathcal{D}$ and each $*$ needs not be isomorphic to each other.

A full subcategory $\mathcal{C} \subset \mathcal{D}$ is called a {\em thick subcategory} if it is a triangulated subcategory which is closed under taking direct summands.
For $E \in \mathcal{D}$, we denote by $\mathrm{thick}(E)$ the smallest thick subcategory containing $E$, and by $\mathrm{add}(E)$ the smallest full additive subcategory containing $E$ and closed under taking direct summands.

Now let $\mathcal{A}$ be an additive category.
An idempotent morphism $e \in \mathrm{Hom}_\mathcal{A}(E,E)$ is said to be {\em split} if there exist $f \in \mathrm{Hom}_\mathcal{A}(E,F)$ and $g \in \mathrm{Hom}_\mathcal{A}(F,E)$ such that $g \circ f = e$ and $f \circ g = \mathrm{id}$.
An additive category is called {\em idempotent complete} if every idempotent morphism therein splits.
In this paper, we mostly assume that our triangulated category $\mathcal{D}$ is idempotent complete.
It is because $\mathcal{D}$ should be idempotent complete in order to admit a bounded t-structure or a bounded co-t-structure with idempotent complete co-heart (see \cite[Theorem]{LC} and \cite[Theorem 2.9]{IY} respectively).

Finally, when we consider the categorical entropy of an endofunctor $\Phi : \mathcal{D} \to \mathcal{D}$, we always assume that $\mathcal{D}$ has a split-generator and $\Phi^N \not\cong 0$ for every $N>0$.

In each section, a triangulated category $\mathcal{D}$ is assumed to satisfy the following:
\begin{center}
\begin{tabular}{|c|c|}
\hline
Section 2 & $\mathbf{K}$-linear (and saturated in Theorem \ref{dhkk-thm})\\
\hline
Sections 3, 4 & $\mathbf{K}$-linear and idempotent complete\\
\hline
Section 5 & $\mathbf{K}$-linear, idempotent complete and algebraic\\
\hline
\end{tabular}
\end{center}

\subsection{Categorical entropy}

In this section, we review the definition and basic properties of categorical entropy.
The reader may refer to \cite{DHKK} in which the notion of categorical entropy was firstly introduced.

Let $\mathcal{D}$ be a $\mathbf{K}$-linear triangulated category.

\begin{dfn}[{\cite[Definition 2.1]{DHKK}}]
Let $E,F \in \mathcal{D}$.
For $t \in \mathbf{R}$, the {\em categorical complexity} of $F$ with respect to $E$ is defined by
\begin{equation*}
\delta_t^\mathcal{D}(E,F) = \inf
\left\{\sum_{i=1}^m e^{k_i t} \,\left|\,
\begin{tikzcd}[column sep=tiny]
0 \ar[rr] & & * \ar[dl] \ar[rr] & & * \ar[dl] & \cdots & * \ar[rr] & & F \oplus F' \ar[dl]\\
& \Sigma^{k_1}E \ar[ul,dashed] & & \Sigma^{k_2}E \ar[ul,dashed] & & & & \Sigma^{k_m}E \ar[ul,dashed] &
\end{tikzcd}
\text{ for some } F' \in \mathcal{D}
\right.\right\}
\end{equation*}
if $F \not\cong 0$, and $\delta_t^\mathcal{D}(E,F) = 0$ if $F \cong 0$.
Note that $\delta_t^\mathcal{D}(E,F) = \infty$ whenever $F \not\in \mathrm{thick}(E)$.
We simply write $\delta_t(E,F) \coloneqq \delta_t^\mathcal{D}(E,F)$ if $\mathcal{D}$ is clear from the context.
\end{dfn}

\begin{lem}[{\cite[Proposition 2.3]{DHKK}}]\label{dhkk-lem}
For $D,E,F \in \mathcal{D}$, the following hold:
\begin{itemize}
\item[(1)] $\delta_t(E,F) \leq \delta_t(E,D) \delta_t(D,F)$.
\item[(2)] $\delta_t(D,E \oplus F) \leq \delta_t(D,E) + \delta_t(D,F)$.
\item[(3)] $\delta_t(\Phi(E),\Phi(F)) \leq \delta_t(E,F)$ for any exact functor $\Phi : \mathcal{D} \to \mathcal{D}'$.
\end{itemize}
\end{lem}

Recall that an object $G \in \mathcal{D}$ is called a {\em split-generator} of $\mathcal{D}$ if $\mathcal{D} = \mathrm{thick}(G)$.
Note that if $G$ is a split-generator of $\mathcal{D}$ then $\delta_t(G,E) < \infty$ for every $E \in \mathcal{D}$.

\begin{dfn}[{\cite[Definition 2.5]{DHKK}}]\label{ent}
Let $G$ be a split-generator of $\mathcal{D}$.
The {\em categorical entropy} of an exact endofunctor $\Phi : \mathcal{D} \to \mathcal{D}$ is defined by
\begin{equation*}
h_t^\mathcal{D}(\Phi) = \lim_{N \to \infty} \frac{1}{N} \log \delta_t^\mathcal{D}(G,\Phi^N(G)).
\end{equation*}
We simply write $h_t(\Phi) \coloneqq h_t^\mathcal{D}(\Phi)$ if $\mathcal{D}$ is clear from the context.
\end{dfn}

\begin{lem}[{\cite[Lemma 2.6]{DHKK}}]
The categorical entropy is well-defined in the sense that it does not depend on the choice of a split-generator and the limit of the sequence $\left\{ \frac{1}{N} \log \delta_t(G,\Phi^N(G)) \right\}_{N=1}^\infty$ exists for any split-generator $G$ of $\mathcal{D}$.
\end{lem}

Let $\Phi : \mathcal{D} \to \mathcal{D}$ be an exact endofunctor.
The categorical entropy of the opposite functor $\Phi^\mathrm{op} : \mathcal{D}^\mathrm{op} \to \mathcal{D}^\mathrm{op}$ can be computed from that of the original functor as follows.

\begin{lem}\label{op-ent}
For an exact endofunctor $\Phi : \mathcal{D} \to \mathcal{D}$, we have
\begin{equation*}
h_t^{\mathcal{D}^\mathrm{op}}(\Phi^\mathrm{op}) = h_{-t}^\mathcal{D}(\Phi).
\end{equation*}
\end{lem}

The following can be easily checked using the octahedral axiom.

\begin{lem}\label{op-lem}
Suppose we have a filtration
\begin{equation*}
\begin{tikzcd}[column sep=tiny]
0 \ar[rr] & & A_1 \ar[dl] \ar[rr,"f_1"] & & A_2 \ar[dl] & \cdots & A_{m-1} \ar[rr,"f_{m-1}"] & & F \ar[dl]\\
& E_1 \ar[ul,dashed] & & E_2 \ar[ul,dashed] & & & & E_m \ar[ul,dashed] &
\end{tikzcd}.
\end{equation*}
Then there exists a diagram
\begin{equation*}
\begin{tikzcd}[column sep=tiny]
F \ar[rr] & & B_1 \ar[dl,dashed] \ar[rr] & & B_2 \ar[dl,dashed] & \cdots & B_{m-1} \ar[rr] & & 0 \ar[dl,dashed]\\
& E_1 \ar[ul] & & E_2 \ar[ul] & & & & E_m \ar[ul] &
\end{tikzcd}
\end{equation*}
consisting of exact triangles where $B_i$ fits into the exact triangle $A_i \overset{f_{m-1} \circ \cdots \circ f_i}{\longrightarrow} F \to B_i \to \Sigma A_i$.
\end{lem}

\begin{proof}[Proof of Lemma \ref{op-ent}]
Recall that the shift functor of the opposite category is given by $\Sigma_{\mathcal{D}^\mathrm{op}} = \Sigma_\mathcal{D}^{-1}$.
Therefore, for every $E,F \in \mathcal{D}$, we see that $\delta_t^{\mathcal{D}^\mathrm{op}}(E,F) = \delta_{-t}^\mathcal{D}(E,F)$ by Lemma \ref{op-lem}.
Then the assertion immediately follows from this.
\end{proof}

A {\em dg enhancement} of $\mathcal{D}$ is a pair $(\mathcal{A},\varepsilon)$ of a pretriangulated dg category $\mathcal{A}$ and an exact equivalence $\varepsilon : H^0(\mathcal{A}) \to \mathcal{D}$.
A triangulated category is called {\em algebraic} if it admits a dg enhancement.
Moreover a triangulated category is called {\em saturated} if it admits a dg enhancement $(\mathcal{A},\varepsilon)$ such that $\mathcal{A}$ is a smooth proper dg category.

For a saturated triangulated category $\mathcal{D}$, the categorical entropy of an exact endofunctor $\Phi : \mathcal{D} \to \mathcal{D}$ can be described as the exponential growth rate of the dimension of $\mathrm{Hom}^*(G,\Phi^N(G))$.

\begin{thm}[{\cite[Theorem 2.7]{DHKK}}]\label{dhkk-thm}
Let $\mathcal{D}$ be a saturated triangulated category and $\Phi : \mathcal{D} \to \mathcal{D}$ be an exact endofunctor.
Then we have
\begin{equation*}
h_t(\Phi) = \lim_{N \to \infty} \frac{1}{N} \log \sum_{k \in \mathbf{Z}} \dim \mathrm{Hom}^{-k}(G,\Phi^N(G))e^{kt}
\end{equation*}
for any split-generator $G$ of $\mathcal{D}$.
\end{thm}

\subsection{Entropy and semiorthogonal decompositions}

In this section, we shall see how the categorical entropy interacts with (weak) semiorthogonal decompositions.
The contents of this section will not be used in the rest of this paper.
So the reader can freely skip this section.

Let $\mathcal{D}$ be a $\mathbf{K}$-linear triangulated category.
A full triangulated subcategory $\mathcal{C} \subset \mathcal{D}$ is called {\em right (resp. left) admissible} if the inclusion functor $i : \mathcal{C} \to \mathcal{D}$ admits a right (resp. left) adjoint functor.
It is known that if $\mathcal{C} \subset \mathcal{D}$ is right (resp. left) admissible then its {\em right (resp. left) orthogonal}
\begin{gather*}
\mathcal{C}^{\perp_\mathcal{D}} \coloneqq \{E \in \mathcal{D} \,|\, \mathrm{Hom}(F,E)=0 \text{ for all } F \in \mathcal{C}\}\\
(\text{resp. } ^{\perp_\mathcal{D}}\mathcal{C} \coloneqq \{E \in \mathcal{D} \,|\, \mathrm{Hom}(E,F)=0 \text{ for all } F \in \mathcal{C}\})
\end{gather*}
is left (resp. right) admissible.
We simply write $\mathcal{C}^\perp \coloneqq \mathcal{C}^{\perp_\mathcal{D}}$ and $^\perp\mathcal{C} \coloneqq ^{\perp_\mathcal{D}}\mathcal{C}$ if $\mathcal{D}$ is clear from the context.

\begin{prop}\label{ent-loc}
Let $\mathcal{C}$ be a right (resp. left) admissible subcategory of $\mathcal{D}$.
Then, for an exact autoequivalence $\Phi : \mathcal{D} \to \mathcal{D}$ such that $\Phi(\mathcal{C}) \subset \mathcal{C}$, we have
\begin{gather*}
h_t^\mathcal{D}(\Phi) = \max\{h_t^\mathcal{C}(\Phi|_\mathcal{C}),h_t^{\mathcal{C}^\perp}(\Phi|_{\mathcal{C}^\perp})\}\\
(\text{resp. } h_t^\mathcal{D}(\Phi) = \max\{h_t^\mathcal{C}(\Phi|_\mathcal{C}),h_t^{^\perp\mathcal{C}}(\Phi|_{^\perp\mathcal{C}})\}).
\end{gather*}
\end{prop}

If $\{a_N\}_{N=1}^\infty$ is a submultiplicative sequence of positive real numbers then $\{\log a_N\}_{N=1}^\infty$ is a subadditive sequence and therefore the limit of the sequence $\left\{\frac{1}{N} \log a_N\right\}_{N=1}^\infty$ exists and can be written as
\begin{equation*}
\lim_{N \to \infty} \frac{1}{N} \log a_N = \inf_{N \geq 1} \frac{1}{N} \log a_N
\end{equation*}
by Fekete's lemma.

The following can be easily verified.

\begin{lem}\label{lim-lem}
Let $\{a_N\}_{N=1}^\infty,\{b_N\}_{N=1}^\infty$ be submultiplicative sequences of positive real numbers.
Then the sequence $\{a_N+b_N\}_{N=1}^\infty$ is also submultiplicative and
\begin{equation*}
\lim_{N \to \infty} \frac{1}{N} \log (a_N+b_N) = \max\left\{\lim_{N \to \infty} \frac{1}{N} \log a_N,\lim_{N \to \infty} \frac{1}{N} \log b_N\right\}.
\end{equation*}
\end{lem}

\begin{proof}[Proof of Proposition \ref{ent-loc}]
Without loss of generality, we assume that $\mathcal{C}$ is right admissible.
Consider the following diagram
\begin{equation*}
\begin{tikzcd}
\mathcal{C} \ar[r,shift right,"i",swap] & \mathcal{D} \ar[r,shift right,"j^L",swap] \ar[l,shift right,"i^R",swap] & \mathcal{C}^\perp \ar[l,shift right,"j",swap]
\end{tikzcd}
\end{equation*}
where $i : \mathcal{C} \to \mathcal{D}$, $j : \mathcal{C}^\perp \to \mathcal{D}$ are the inclusion functors and $i^R : \mathcal{D} \to \mathcal{C}$, $j^L : \mathcal{D} \to \mathcal{C}^\perp$ are their adjoint functors.
Note that $j^L \circ i = 0$ and $i^R \circ j = 0$.

Now let $G_\mathcal{C}$ and $G_{\mathcal{C}^\perp}$ be split-generators of $\mathcal{C}$ and $\mathcal{C}^\perp$ respectively.
Then clearly
\begin{equation*}
G \coloneqq i(G_\mathcal{C}) \oplus j(G_{\mathcal{C}^\perp})
\end{equation*}
is a split-generator of $\mathcal{D}$.

(Step 1)
Since $i^R(\Phi^N(G)) = i^R(\Phi^N(i(G_\mathcal{C}))) = \Phi|_\mathcal{C}^N(G_\mathcal{C})$, we have
\begin{equation*}
\delta_t^\mathcal{D}(G,\Phi^N(G)) \geq \delta_t^\mathcal{C}(i^R(G),i^R(\Phi^N(G))) = \delta_t^\mathcal{C}(G_\mathcal{C},\Phi|_\mathcal{C}^N(G_\mathcal{C}))
\end{equation*}
by Lemma \ref{dhkk-lem} (3).
This shows that $h_t^\mathcal{D}(\Phi) \geq h_t^\mathcal{C}(\Phi|_\mathcal{C})$.
Similarly, we get $h_t^\mathcal{D}(\Phi) \geq h_t^{\mathcal{C}^\perp}(\Phi|_{\mathcal{C}^\perp})$.

(Step 2)
Using Lemma \ref{dhkk-lem} (2), (3), we obtain
\begin{align*}
\delta_t^\mathcal{D}(G,\Phi^N(G))
&\leq \delta_t^\mathcal{D}(G,\Phi^N(i(G_\mathcal{C}))) + \delta_t^\mathcal{D}(G,\Phi^N(j(G_{\mathcal{C}^\perp})))\\
&\leq \delta_t^\mathcal{D}(i(G_\mathcal{C}),\Phi^N(i(G_\mathcal{C}))) + \delta_t^\mathcal{D}(j(G_{\mathcal{C}^\perp}),\Phi^N(j(G_{\mathcal{C}^\perp})))\\
&= \delta_t^\mathcal{D}(i(G_\mathcal{C}),i(\Phi|_\mathcal{C}^N(G_\mathcal{C}))) + \delta_t^\mathcal{D}(j(G_{\mathcal{C}^\perp}),j(\Phi|_{\mathcal{C}^\perp}^N(G_{\mathcal{C}^\perp})))\\
&\leq \delta_t^\mathcal{C}(G_\mathcal{C},\Phi|_\mathcal{C}^N(G_\mathcal{C})) + \delta_t^{\mathcal{C}^\perp}(G_{\mathcal{C}^\perp},\Phi|_{\mathcal{C}^\perp}^N(G_{\mathcal{C}^\perp})).
\end{align*}
As $\{\delta_t^\mathcal{C}(G_\mathcal{C},\Phi|_\mathcal{C}^N(G_\mathcal{C}))\}_{N=1}^\infty$ and $\{\delta_t^{\mathcal{C}^\perp}(G_{\mathcal{C}^\perp},\Phi|_{\mathcal{C}^\perp}^N(G_{\mathcal{C}^\perp}))\}_{N=1}^\infty$ are submultiplicative, we can apply Lemma \ref{lim-lem} and get $h_t^\mathcal{D}(\Phi) \leq \max\{h_t^\mathcal{C}(\Phi|_\mathcal{C}),h_t^{\mathcal{C}^\perp}(\Phi|_{\mathcal{C}^\perp})\}$.
\end{proof}

Let us recall the definition of a (weak) semiorthogonal decomposition.

\begin{dfn}[{\cite[Definition 4.1]{BK},\cite[Definition 5]{Orl}}]
A sequence $(\mathcal{C}_1,\dots,\mathcal{C}_n)$ of full triangulated subcategories of $\mathcal{D}$ is called a {\em weak semiorthogonal decomposition} of $\mathcal{D}$ if there exist left admissible subcategories $\mathcal{C}_1 \eqqcolon \mathcal{D}_1 \subset \mathcal{D}_2 \subset \cdots \subset \mathcal{D}_n \coloneqq \mathcal{D}$ such that $\mathcal{C}_i^{\perp_{\mathcal{D}_i}} = \mathcal{D}_{i-1}$ for every $1 < i \leq n$.
In this case, we write $\mathcal{D} = \langle\mathcal{C}_1,\dots,\mathcal{C}_n\rangle$.
If $\mathcal{D}_i$'s are also right admissible then $\mathcal{D} = \langle\mathcal{C}_1,\dots,\mathcal{C}_n\rangle$ is called a {\em semiorthogonal decomposition} of $\mathcal{D}$.
\end{dfn}

\begin{cor}\label{sod}
Suppose $\mathcal{D}$ admits a weak semiorthogonal decomposition $\langle\mathcal{C}_1,\dots,\mathcal{C}_n\rangle$.
Then, for an exact autoequivalence $\Phi : \mathcal{D} \to \mathcal{D}$ such that $\Phi(\mathcal{C}_i) \subset \mathcal{C}_i$ for every $1 \leq i \leq n$, we have
\begin{equation*}
h_t^\mathcal{D}(\Phi) = \max_{1 \leq i \leq n} h_t^{\mathcal{C}_i}(\Phi|_{\mathcal{C}_i}).
\end{equation*}
\end{cor}

\begin{proof}
Take left admissible subcategories $\mathcal{C}_1 \eqqcolon \mathcal{D}_1 \subset \mathcal{D}_2 \subset \cdots \subset \mathcal{D}_n \coloneqq \mathcal{D}$ such that $\mathcal{C}_i^{\perp_{\mathcal{D}_i}} = \mathcal{D}_{i-1}$ for every $1 < i \leq n$.
Note that $\mathcal{C}_i = ^{\perp_{\mathcal{D}_i}}\mathcal{D}_{i-1} \subset \mathcal{D}_i$ is right admissible for every $1 < i \leq n$.
Therefore we see that
\begin{align*}
h_t^\mathcal{D}(\Phi)
&= \max\{h_t^{\mathcal{C}_n}(\Phi|_{\mathcal{C}_n}),h_t^{\mathcal{D}_{n-1}}(\Phi|_{\mathcal{D}_{n-1}})\}\\
&= \max\{h_t^{\mathcal{C}_n}(\Phi|_{\mathcal{C}_n}),h_t^{\mathcal{C}_{n-1}}(\Phi|_{\mathcal{C}_{n-1}}),h_t^{\mathcal{D}_{n-2}}(\Phi|_{\mathcal{D}_{n-2}})\}\\
&\;\;\vdots\\
&= \max_{1 \leq i \leq n} h_t^{\mathcal{C}_i}(\Phi|_{\mathcal{C}_i})
\end{align*}
by applying Proposition \ref{ent-loc} iteratively.
\end{proof}

\section{Entropy and t-structures}

\subsection{t-structures}

Let us first review the definition and basic properties of t-structures.
The reader may refer to \cite{Jor} for a brief overview of this subject.

Let $\mathcal{D}$ be a $\mathbf{K}$-linear idempotent complete triangulated category.

\begin{dfn}[{\cite[D\'efinition 1.3.1]{BBD}}]\label{t-str}
A {\em t-structure} on $\mathcal{D}$ is a pair $\mathfrak{a} = (\mathcal{D}^{\leq 0},\mathcal{D}^{\geq 0})$ of full additive subcategories of $\mathcal{D}$ satisfying the following conditions:
\begin{enumerate}
\item $\Sigma\mathcal{D}^{\leq 0} \subset \mathcal{D}^{\leq 0}$ and $\Sigma^{-1}\mathcal{D}^{\geq 0} \subset \mathcal{D}^{\geq 0}$.
\item $\mathrm{Hom}(E,F) = 0$ if $E \in \mathcal{D}^{\leq 0},F \in \mathcal{D}^{\geq 1}$ where $\mathcal{D}^{\geq 1} \coloneqq \Sigma^{-1}\mathcal{D}^{\geq 0}$.
\item For every $E \in \mathcal{D}$, there exists an exact triangle $D \to E \to F \to \Sigma D$ with $D \in \mathcal{D}^{\leq 0},F \in \mathcal{D}^{\geq 1}$.
\end{enumerate}
If they further satisfy
\begin{equation*}
\mathcal{D} = \bigcup_{k \in \mathbf{Z}} \Sigma^k\mathcal{D}^{\leq 0} = \bigcup_{k \in \mathbf{Z}} \Sigma^k\mathcal{D}^{\geq 0}
\end{equation*}
then $\mathfrak{a}$ is called {\em bounded}.
\end{dfn}

Let $\mathcal{A}$ be an abelian category.
A non-zero object of $\mathcal{A}$ is called {\em simple} if its only subobjects are 0 or itself.
A {\em composition series} of $E \in \mathcal{A}$ is a finite series of subobjects
\begin{equation*}
0 = E_0 \subset E_1 \subset \cdots \subset E_m = E
\end{equation*}
such that $E_i/E_{i-1}$ is simple for every $1 \leq i \leq m$.
In this case, the integer $m \geq 0$, called the {\em length} of $E$, does not depend on the choice of a composition series.
Moreover the simple quotients $E_1/E_0,\dots,E_m/E_{m-1}$, called the {\em composition factors} of $E$, do not depend on the choice of a composition series up to isomorphisms.
An abelian category is called a {\em length category} if every its object has a composition series.

\begin{dfn}
Let $\mathfrak{a} = (\mathcal{D}^{\leq 0},\mathcal{D}^{\geq 0})$ be a t-structure on $\mathcal{D}$.
The subcategory $\mathcal{H}_\mathfrak{a} \coloneqq \mathcal{D}^{\leq 0} \cap \mathcal{D}^{\geq 0}$, called the {\em heart} of $\mathfrak{a}$, is an abelian category.
We say $\mathcal{H}_\mathfrak{a}$ is {\em finite} if it is a length category and only has finitely many isomorphism classes of simple objects.
If the heart $\mathcal{H}_\mathfrak{a}$ is finite then, for $E \in \mathcal{H}_\mathfrak{a}$, we define $\lambda_{\mathcal{H}_\mathfrak{a}}(E) \geq 0$ as the length of a composition series of $E$ in $\mathcal{H}_\mathfrak{a}$.
\end{dfn}

The object $D \in \mathcal{D}^{\leq 0}$ (resp. $F \in \mathcal{D}^{\geq 1}$) in Definition \ref{t-str} (3) depends functorially on $E$ and so defines the {\em truncation functor} $\tau_\mathfrak{a}^{\leq 0} : \mathcal{D} \to \mathcal{D}^{\leq 0}$ (resp. $\tau_\mathfrak{a}^{\geq 1} : \mathcal{D} \to \mathcal{D}^{\geq 1}$).
Define $\tau^{\geq 0} \coloneqq \Sigma \circ \tau^{\geq 1} \circ \Sigma^{-1} : \mathcal{D} \to \mathcal{D}^{\geq 0}$.
For $k \in \mathbf{Z}$, the {\em $k$th cohomology functor} with respect to $\mathfrak{a}$ is defined by
\begin{equation*}
H_\mathfrak{a}^k \coloneqq \tau^{\leq 0} \circ \tau^{\geq 0} \circ \Sigma^k : \mathcal{D} \to \mathcal{H}_\mathfrak{a}.
\end{equation*}

A t-structure provides a nice way to find a filtration of an arbitrary object.

\begin{prop}[{\cite[Section 1.3.12]{BBD}}]\label{t-filt}
Let $\mathfrak{a}$ be a bounded t-structure on $\mathcal{D}$.
Then, for any $E \in \mathcal{D}$, there exists a filtration
\begin{equation*}
\begin{tikzcd}[column sep=tiny]
0 \ar[rr] & & * \ar[dl] \ar[rr] & & * \ar[dl] & \cdots & * \ar[rr] & & E \ar[dl]\\
& \Sigma^{k_1}E_1 \ar[ul,dashed] & & \Sigma^{k_2}E_2 \ar[ul,dashed] & & & & \Sigma^{k_m}E_m \ar[ul,dashed] &
\end{tikzcd}
\end{equation*}
such that $E_i \in \mathcal{H}_\mathfrak{a}$ and $k_1>\cdots>k_m$.
\end{prop}

\begin{rmk}
A filtration of $E$ as in Proposition \ref{t-filt} is determined uniquely up to isomorphisms.
In fact, there are isomorphisms $E_i \cong H_\mathfrak{a}^{k_i}(E)$ for all $1 \leq i \leq m$.
In this paper, a filtration of $E$ as in Proposition \ref{t-filt} will be called a {\em t-filtration} of $E$ with respect to $\mathfrak{a}$.
\end{rmk}

\subsection{Entropy via t-structures}

The goal of this section is to define an entropy-type invariant using a bounded t-structure with finite heart and study its relationship with the categorical entropy.

Let $\mathcal{D}$ be a $\mathbf{K}$-linear idempotent complete triangulated category.

\begin{dfn}
Let $\mathfrak{a}$ be a bounded t-structure on $\mathcal{D}$ with finite heart.
For $E \in \mathcal{D}$, let
\begin{equation*}
\begin{tikzcd}[column sep=tiny]
0 \ar[rr] & & * \ar[dl] \ar[rr] & & * \ar[dl] & \cdots & * \ar[rr] & & E \ar[dl]\\
& \Sigma^{k_1}E_1 \ar[ul,dashed] & & \Sigma^{k_2}E_2 \ar[ul,dashed] & & & & \Sigma^{k_m}E_m \ar[ul,dashed] &
\end{tikzcd}
\end{equation*}
be its t-filtration with respect to $\mathfrak{a}$.
Then, for $t \in \mathbf{R}$, we define
\begin{equation*}
\hat{\delta}_t^{\mathcal{D},\mathfrak{a}}(E) \coloneqq \sum_{i=1}^m \lambda_{\mathcal{H}_\mathfrak{a}}(E_i) e^{k_i t}.
\end{equation*}
We simply write $\hat{\delta}_t^\mathfrak{a}(E) \coloneqq \hat{\delta}_t^{\mathcal{D},\mathfrak{a}}(E)$ if $\mathcal{D}$ is clear from the context.
\end{dfn}

\begin{dfn}\label{t-ent}
Let $\mathfrak{a}$ be a bounded t-structure on $\mathcal{D}$ with finite heart and $\Phi : \mathcal{D} \to \mathcal{D}$ be an exact endofunctor.
Then we define
\begin{equation*}
\hat{h}_t^{\mathcal{D},\mathfrak{a}}(\Phi) \coloneqq \sup_{E \in \mathcal{D}} \limsup_{N \to \infty} \frac{1}{N} \log \hat{\delta}_t^{\mathcal{D},\mathfrak{a}}(\Phi^N(E)).
\end{equation*}
We simply write $\hat{h}_t^\mathfrak{a}(\Phi) \coloneqq \hat{h}_t^{\mathcal{D},\mathfrak{a}}(\Phi)$ if $\mathcal{D}$ is clear from the context.
\end{dfn}

The definition of $\hat{h}_t^\mathfrak{a}$ is an analogue of that of the {\em mass growth} which is defined using the mass function of a stability condition instead of $\hat{\delta}_t^\mathfrak{a}$ (see \cite[Section 4.5]{DHKK} and \cite[Definition 3.2]{Ike}).

The following is one of the main theorems of this paper.
It particularly states that $\hat{h}_t^\mathfrak{a}$ coincides with the categorical entropy.

\begin{thm}\label{t-thm-1}
Let $\mathfrak{a}$ be a bounded t-structure on $\mathcal{D}$ with finite heart and $\Phi : \mathcal{D} \to \mathcal{D}$ be an exact endofunctor.
Then we have
\begin{equation*}
h_t(\Phi) = \hat{h}_t^\mathfrak{a}(\Phi) = \lim_{N \to \infty} \frac{1}{N} \log \hat{\delta}_t^\mathfrak{a}(\Phi^N(G))
\end{equation*}
for any split-generator $G$ of $\mathcal{D}$.
In particular, $\hat{h}_t^\mathfrak{a}(\Phi)$ does not depend on the choice of $\mathfrak{a}$.
\end{thm}

When $\mathcal{D}$ admits a bounded t-structure with finite heart, Theorem \ref{t-thm-1} provides an easier way to compute the categorical entropy.
Indeed it allows us to compute the categorical entropy using an essentially uniquely determined filtration (i.e., a t-filtration) of $\Phi^N(G)$ instead of all possible filtrations of any objects having $\Phi^N(G)$ as a direct summand.

Many parts of the proof of Theorem \ref{t-thm-1} is motivated by that of \cite[Theorem 3.5]{Ike} as we will see shortly.
We shall prove Theorem \ref{t-thm-1} in the following three steps:

(Step 1)
$\displaystyle \hat{h}_t^\mathfrak{a}(\Phi) = \limsup_{N \to \infty} \frac{1}{N} \log \hat{\delta}_t^\mathfrak{a}(\Phi^N(G))$ for any split-generator $G$ of $\mathcal{D}$.

(Step 2)
$\hat{h}_t^\mathfrak{a}(\Phi) = h_t(\Phi)$.

(Step 3)
The limit of the sequence $\left\{ \frac{1}{N} \log \hat{\delta}_t^\mathfrak{a}(\Phi^N(G)) \right\}_{N=1}^\infty$ exists for any split-generator $G$ of $\mathcal{D}$.

The proof will be given at the end of this section after showing some lemmas and propositions.
Along the way, we frequently use the following facts which can be checked easily using the octahedral axiom.

\begin{lem}\label{ex-lem-1}
Suppose we have two exact triangles $A \to B \to C \overset{g}{\to} \Sigma A$ and $E \overset{f}{\to} C \to F \to \Sigma E$.
Then there exists a diagram
\begin{equation*}
\begin{tikzcd}[column sep=tiny]
A \ar[rr] & & * \ar[dl] \ar[rr] & & B \ar[dl]\\
& E \ar[ul,dashed,"g \circ f"] & & F \ar[ul,dashed] &
\end{tikzcd}
\end{equation*}
consisting of exact triangles.
\end{lem}

\begin{lem}\label{ex-lem-2}
Suppose we have a diagram
\begin{equation*}
\begin{tikzcd}[column sep=tiny]
A \ar[rr] & & * \ar[dl,"g"] \ar[rr] & & B \ar[dl]\\
& E \ar[ul,dashed] & & F \ar[ul,dashed,"f"] &
\end{tikzcd}
\end{equation*}
consisting of exact triangles.
Then there exists an exact triangle $A \to B \to C \to \Sigma A$ where $C$ fits into the exact triangle $E \to C \to F \overset{\Sigma g \circ f}{\longrightarrow} \Sigma E$.
\end{lem}

\begin{lem}
Consider a filtration
\begin{equation*}
\begin{tikzcd}[column sep=tiny]
0 \ar[rr] & & * \ar[dl] \ar[rr] & & * \ar[dl] & \cdots & * \ar[rr] & & E \ar[dl]\\
& E_1 \ar[ul,dashed] & & E_2 \ar[ul,dashed] & & & & E_m \ar[ul,dashed] &
\end{tikzcd}
\end{equation*}
of $E \in \mathcal{D}$.
If there are $E'_1,\dots,E'_m \in \mathcal{D}$ such that $E'_i \cong E_i$ for every $1 \leq i \leq m$ then there exists a filtration of $E$ of the form
\begin{equation*}
\begin{tikzcd}[column sep=tiny]
0 \ar[rr] & & * \ar[dl] \ar[rr] & & * \ar[dl] & \cdots & * \ar[rr] & & E \ar[dl]\\
& E'_1 \ar[ul,dashed] & & E'_2 \ar[ul,dashed] & & & & E'_m \ar[ul,dashed] &
\end{tikzcd}.
\end{equation*}
\end{lem}

The following is the key proposition to prove Theorem \ref{t-thm-1}.

\begin{prop}\label{t-ex-prop}
Let $\mathfrak{a}$ be a bounded t-structure on $\mathcal{D}$ with finite heart.
If there is an exact triangle $D \to E \to F \to \Sigma D$ in $\mathcal{D}$ then
\begin{equation*}
\hat{\delta}_t^\mathfrak{a}(E) \leq \hat{\delta}_t^\mathfrak{a}(D) + \hat{\delta}_t^\mathfrak{a}(F).
\end{equation*}
\end{prop}

\begin{proof}
Take t-filtrations
\begin{equation*}
\begin{tikzcd}[column sep=tiny]
0 \ar[rr] & & * \ar[dl] \ar[rr] & & * \ar[dl] & \cdots & * \ar[rr] & & D \ar[dl]\\
& \Sigma^{k_1}D_1 \ar[ul,dashed] & & \Sigma^{k_2}D_2 \ar[ul,dashed] & & & & \Sigma^{k_m}D_m \ar[ul,dashed] &
\end{tikzcd}
\end{equation*}
and
\begin{equation*}
\begin{tikzcd}[column sep=tiny]
0 \ar[rr] & & * \ar[dl] \ar[rr] & & * \ar[dl] & \cdots & * \ar[rr] & & F \ar[dl]\\
& \Sigma^{k'_1}F_1 \ar[ul,dashed] & & \Sigma^{k'_2}F_2 \ar[ul,dashed] & & & & \Sigma^{k'_{m'}}F_{m'} \ar[ul,dashed] &
\end{tikzcd}
\end{equation*}
with respect to $\mathfrak{a}$.
Inserting these filtrations into the exact triangle $D \to E \to F \to \Sigma D$ using Lemma \ref{ex-lem-1}, we get a filtration
\begin{equation*}
\begin{tikzcd}[column sep=tiny]
0 \ar[rr] & & * \ar[dl] & \cdots & * \ar[rr] & & * \ar[dl] \ar[rr] & & * \ar[dl] & \cdots & * \ar[rr] &  & E \ar[dl]\\
& \Sigma^{k_1}D_1 \ar[ul,dashed] & & & & \Sigma^{k_m}D_m \ar[ul,dashed] & & \Sigma^{k'_1}F_1 \ar[ul,dashed] & & & & \Sigma^{k'_{m'}}F_{m'} \ar[ul,dashed] &
\end{tikzcd}.
\end{equation*}
For convenience, we write it as
\begin{equation}\label{not-t-filt}
\begin{tikzcd}[column sep=tiny]
0 \ar[rr] & & * \ar[dl] & \cdots & * \ar[rr] & & * \ar[dl] \ar[rr] & & * \ar[dl] & \cdots & * \ar[rr] &  & E \ar[dl]\\
& \Sigma^{k''_1}E_1 \ar[ul,dashed] & & & & \Sigma^{k''_m}E_m \ar[ul,dashed] & & \Sigma^{k''_{m+1}}E_{m+1} \ar[ul,dashed] & & & & \Sigma^{k''_{m''}}E_{m''} \ar[ul,dashed] &
\end{tikzcd}
\end{equation}
where $E_1 \coloneqq D_1,\dots,E_m \coloneqq D_m,E_{m+1} \coloneqq F_1,\dots,E_{m''} \coloneqq F_{m'}$ and $k''_1 \coloneqq k_1,\dots,k''_m \coloneqq k_m,k''_{m+1} \coloneqq k_1',\dots,k''_{m''} \coloneqq k'_{m'}$.

In general, the filtration \eqref{not-t-filt} is not a t-filtration, i.e., there may be some $1 \leq i < m''$ such that $k''_i \leq k''_{i+1}$.
However we can produce a t-filtration from the filtration \eqref{not-t-filt} through the following procedure.
Suppose we actually have $k''_i \leq k''_{i+1}$ for some $1 \leq i < m''$ and consider the following part of the filtration \eqref{not-t-filt}
\begin{equation}\label{not-t-part}
\begin{tikzcd}[column sep=tiny]
A \ar[rr] & & B \ar[dl,"g"] \ar[rr] & & C \ar[dl]\\
& \Sigma^{k''_i}E_i \ar[ul,dashed] & & \Sigma^{k''_{i+1}}E_{i+1} \ar[ul,dashed,"f"] &
\end{tikzcd}.
\end{equation}

(Case 1)
Suppose $k''_i - k''_{i+1} < -1$.
Then $\Sigma g \circ f \in \mathrm{Hom}(\Sigma^{k''_{i+1}}E_{i+1},\Sigma^{k''_i+1}E_i) = \mathrm{Hom}(E_{i+1},\Sigma^{k''_i-k''_{i+1}+1}E_i) = 0$ since $E_i,E_{i+1} \in \mathcal{H}_\mathfrak{a}$ and $k''_i-k''_{i+1}+1<0$.
Therefore, by Lemma \ref{ex-lem-2}, we obtain an exact triangle $A \to C \to \Sigma^{k''_i}E_i \oplus \Sigma^{k''_{i+1}}E_{i+1} \to \Sigma A$ and so we can replace \eqref{not-t-part} by
\begin{equation*}
\begin{tikzcd}[column sep=tiny]
A \ar[rr] & & * \ar[dl] \ar[rr] & & C \ar[dl]\\
& \Sigma^{k''_{i+1}}E_{i+1} \ar[ul,dashed] & & \Sigma^{k''_i}E_i \ar[ul,dashed] &
\end{tikzcd}
\end{equation*}
using Lemma \ref{ex-lem-1}.

(Case 2)
Suppose $k''_i - k''_{i+1} = 0$.
Then, by Lemma \ref{ex-lem-2}, we obtain an exact triangle $A \to C \to \Sigma^{k''_i}E' \to \Sigma A$ where $E'$ fits into the exact triangle $E_i \to E' \to E_{i+1} \to \Sigma E_i$.
Since $\mathcal{H}_\mathfrak{a}$ is extension-closed, we have $E' \in \mathcal{H}_\mathfrak{a}$.
Thus we can replace \eqref{not-t-part} by
\begin{equation*}
\begin{tikzcd}[column sep=tiny]
A \ar[rr] & & C \ar[dl]\\
& \Sigma^{k''_i}E' \ar[ul,dashed] &
\end{tikzcd}.
\end{equation*}

Note that, by the short exact sequence $0 \to E_i \to E' \to E_{i+1} \to 0$ in $\mathcal{H}_\mathfrak{a}$, we have
\begin{equation*}
\lambda_{\mathcal{H}_\mathfrak{a}}(E') = \lambda_{\mathcal{H}_\mathfrak{a}}(E_i) + \lambda_{\mathcal{H}_\mathfrak{a}}(E_{i+1}).
\end{equation*}

(Case 3)
Suppose $k''_i - k''_{i+1} = -1$.
Then, by Lemma \ref{ex-lem-2}, we obtain an exact triangle $A \to C \to E' \to \Sigma A$ where $E'$ fits into the exact triangle $\Sigma^{k''_i}E_i \to E' \to \Sigma^{k''_i+1}E_{i+1} \to \Sigma^{k''_i+1}E_i$.
By taking the $\mathfrak{a}$-cohomology, we see that $H_\mathfrak{a}^{-j}(E') \neq 0$ if and only if $j=k''_i$ or $k''_i+1 (= k''_{i+1})$.
Therefore $E'$ has the following t-filtration
\begin{equation*}
\begin{tikzcd}[column sep=tiny]
0 \ar[rr] & & * \ar[dl] \ar[rr] & & E' \ar[dl]\\
& \Sigma^{k''_i+1}H_\mathfrak{a}^{-k''_i-1}(E') \ar[ul,dashed] & & \Sigma^{k''_i}H_\mathfrak{a}^{-k''_i}(E') \ar[ul,dashed] &
\end{tikzcd}
\end{equation*}
and so we can replace \eqref{not-t-part} by
\begin{equation*}
\begin{tikzcd}[column sep=tiny]
A \ar[rr] & & * \ar[dl] \ar[rr] & & C \ar[dl]\\
& \Sigma^{k''_i+1}H_\mathfrak{a}^{-k''_i-1}(E') \ar[ul,dashed] & & \Sigma^{k''_i}H_\mathfrak{a}^{-k''_i}(E') \ar[ul,dashed] &
\end{tikzcd}
\end{equation*}
using Lemma \ref{ex-lem-1}.

On the other hand, we have an exact sequence
\begin{equation*}
0 \to H_\mathfrak{a}^{-k''_i-1}(E') \overset{g}{\to} E_{i+1} \to E_i \overset{f}{\to} H_\mathfrak{a}^{-k''_i}(E') \to 0
\end{equation*}
in $\mathcal{H}_\mathfrak{a}$ which can be split into two short exact sequences
\begin{gather*}
0 \to \mathrm{Ker}(f) \to E_i \overset{f}{\to} H_\mathfrak{a}^{-k''_i}(E') \to 0,\\
0 \to H_\mathfrak{a}^{-k''_i-1}(E') \overset{g}{\to} E_{i+1} \to \mathrm{Coker}(g) \to 0.
\end{gather*}
This implies that
\begin{gather*}
\lambda_{\mathcal{H}_\mathfrak{a}}(E_i) = \lambda_{\mathcal{H}_\mathfrak{a}}(H_\mathfrak{a}^{-k''_i}(E')) + \lambda_{\mathcal{H}_\mathfrak{a}}(\mathrm{Ker}(f)) \geq \lambda_{\mathcal{H}_\mathfrak{a}}(H_\mathfrak{a}^{-k''_i}(E')),\\
\lambda_{\mathcal{H}_\mathfrak{a}}(E_{i+1}) = \lambda_{\mathcal{H}_\mathfrak{a}}(H_\mathfrak{a}^{-k''_i-1}(E')) + \lambda_{\mathcal{H}_\mathfrak{a}}(\mathrm{Coker}(g)) \geq \lambda_{\mathcal{H}_\mathfrak{a}}(H_\mathfrak{a}^{-k''_i-1}(E')).
\end{gather*}

The assertion then follows from the above arguments because (Case 1) and (Case 2) do not change $\sum_{i=1}^{m''} \lambda_{\mathcal{H}_\mathfrak{a}}(E_i) e^{k''_i t}$ and (Case 3) makes it smaller.
\end{proof}

The following is an analogue of \cite[Proposition 3.4]{Ike}.

\begin{prop}\label{t-upper}
Let $\mathfrak{a}$ be a bounded t-structure on $\mathcal{D}$ with finite heart.
Then
\begin{equation*}
\hat{\delta}_t^\mathfrak{a}(E) \leq \hat{\delta}_t^\mathfrak{a}(D) \delta_t(D,E)
\end{equation*}
for any objects $D,E \in \mathcal{D}$.
\end{prop}

\begin{proof}
If $E \not\in \mathrm{thick}(D)$ then $\delta_t(D,E) = \infty$ so the inequality obviously holds.
Now assume that $E \in \mathrm{thick}(D)$.
Then, for any $\varepsilon>0$, there is a filtration
\begin{equation}\label{t-filt-prop}
\begin{tikzcd}[column sep=tiny]
0 \ar[rr] & & * \ar[dl] \ar[rr] & & * \ar[dl] & \cdots & * \ar[rr] & & E \oplus E' \ar[dl]\\
& \Sigma^{k_1}D \ar[ul,dashed] & & \Sigma^{k_2}D \ar[ul,dashed] & & & & \Sigma^{k_m}D \ar[ul,dashed] &
\end{tikzcd}
\end{equation}
such that
\begin{equation*}
\sum_{i=1}^m e^{k_it} < \delta_t(D,E) + \varepsilon
\end{equation*}
by the definition of $\delta_t(D,E)$.
Using this inequality, we obtain
\begin{equation*}
\hat{\delta}_t^\mathfrak{a}(E) \leq \hat{\delta}_t^\mathfrak{a}(E \oplus E') \leq \sum_{i=1}^m \hat{\delta}_t^\mathfrak{a}(\Sigma^{k_i}D) = \hat{\delta}_t^\mathfrak{a}(D) \sum_{i=1}^m e^{k_it} < \hat{\delta}_t^\mathfrak{a}(D) \delta_t(D,E) + \varepsilon \cdot \hat{\delta}_t^\mathfrak{a}(D)
\end{equation*}
where the second inequality can be obtained by applying Proposition \ref{t-ex-prop} to every exact triangle in the filtration \eqref{t-filt-prop}.
Since $\varepsilon>0$ is arbitrary, the assertion follows.
\end{proof}

The following is an analogue of \cite[Lemma 3.12]{Ike}.

\begin{lem}\label{t-lower}
Let $\mathfrak{a}$ be a bounded t-structure on $\mathcal{D}$ with finite heart.
Let $S_1,\dots,S_n \in \mathcal{H}_\mathfrak{a}$ be a complete set of pairwise non-isomorphic simple objects and set $S \coloneqq S_1 \oplus \cdots \oplus S_n$.
Then for any $E \in \mathcal{D}$, we have
\begin{equation*}
\hat{\delta}_t^\mathfrak{a}(E) \geq \delta_t(S,E).
\end{equation*}
\end{lem}

\begin{proof}
Take a t-filtration
\begin{equation*}
\begin{tikzcd}[column sep=tiny]
0 \ar[rr] & & * \ar[dl] \ar[rr] & & * \ar[dl] & \cdots & * \ar[rr] & & E \ar[dl]\\
& \Sigma^{k_1}E_1 \ar[ul,dashed] & & \Sigma^{k_2}E_2 \ar[ul,dashed] & & & & \Sigma^{k_m}E_m \ar[ul,dashed] &
\end{tikzcd}.
\end{equation*}
Using composition series of $E_1,\dots,E_m \in \mathcal{H}_\mathfrak{a}$, we can produce a filtration
\begin{equation*}
\begin{tikzcd}[column sep=tiny]
0 \ar[rr] & & * \ar[dl] & \cdots & * \ar[rr] & & * \ar[dl] & \cdots & * \ar[rr] & & * \ar[dl] & \cdots & * \ar[rr] &  & E \ar[dl]\\
& \Sigma^{k_1}S_1^{(1)} \ar[ul,dashed] & & & & \Sigma^{k_1}S_1^{(\lambda_{\mathcal{H}_\mathfrak{a}}(E_1))} \ar[ul,dashed] & & & & \Sigma^{k_m}S_m^{(1)} \ar[ul,dashed] & & & & \Sigma^{k_m}S_m^{(\lambda_{\mathcal{H}_\mathfrak{a}}(E_m))} \ar[ul,dashed] &
\end{tikzcd}
\end{equation*}
using Lemma \ref{ex-lem-1} where $S_i^{(j)} \in \{S_1,\dots,S_n\}$ are the composition factors of $E_i$.
From this, we obtain a filtration
\begin{equation*}
\begin{tikzcd}[column sep=tiny]
0 \ar[rr] & & * \ar[dl] & \cdots & * \ar[rr] & & * \ar[dl] & \cdots & * \ar[rr] & & * \ar[dl] & \cdots & * \ar[rr] &  & E \oplus S' \ar[dl]\\
& \Sigma^{k_1}S \ar[ul,dashed] & & & & \Sigma^{k_1}S \ar[ul,dashed] & & & & \Sigma^{k_m}S \ar[ul,dashed] & & & & \Sigma^{k_m}S \ar[ul,dashed] &
\end{tikzcd}
\end{equation*}
for some $S' \in \mathcal{D}$ which can be written as a direct sum of some shifts of $S_1,\dots,S_n$.
This shows that
\begin{equation*}
\hat{\delta}_t^\mathfrak{a}(E) = \sum_{i=1}^m \lambda_{\mathcal{H}_\mathfrak{a}}(E_i) e^{k_i t} \geq \delta_t(S,E)
\end{equation*}
where the second inequality follows from the definition of $\delta_t(S,E)$.
\end{proof}

\begin{proof}[Proof of Theorem \ref{t-thm-1}]
(Step 1)
Let us first show that
\begin{equation*}
\hat{h}_t^\mathfrak{a}(\Phi) = \limsup_{N \to \infty} \frac{1}{N} \log \hat{\delta}_t^\mathfrak{a}(\Phi^N(G))
\end{equation*}
for any split-generator $G$ of $\mathcal{D}$.
Fix a split-generator $G$ of $\mathcal{D}$.
By definition, we have
\begin{equation*}
\limsup_{N \to \infty} \frac{1}{N} \log \hat{\delta}_t^\mathfrak{a}(\Phi^N(G)) \leq \hat{h}_t^\mathfrak{a}(\Phi).
\end{equation*}
On the other hand, Lemma \ref{dhkk-lem} (3) and Proposition \ref{t-upper} show that
\begin{equation*}
\hat{\delta}_t^\mathfrak{a}(\Phi^N(E)) \leq \hat{\delta}_t^\mathfrak{a}(\Phi^N(G)) \delta_t(\Phi^N(G),\Phi^N(E)) \leq \hat{\delta}_t^\mathfrak{a}(\Phi^N(G)) \delta_t(G,E)
\end{equation*}
for any $E \in \mathcal{D}$.
Taking $\limsup_{N \to \infty} \frac{1}{N} \log(-)$, we obtain
\begin{equation*}
\limsup_{N \to \infty} \frac{1}{N} \log \hat{\delta}_t^\mathfrak{a}(\Phi^N(E)) \leq \limsup_{N \to \infty} \frac{1}{N} \log \hat{\delta}_t^\mathfrak{a}(\Phi^N(G))
\end{equation*}
for any $E \in \mathcal{D}$.
Therefore
\begin{equation*}
\hat{h}_t^\mathfrak{a}(\Phi) \leq \limsup_{N \to \infty} \frac{1}{N} \log \hat{\delta}_t^\mathfrak{a}(\Phi^N(G)).
\end{equation*}

(Step 2)
Let us next show that $\hat{h}_t^\mathfrak{a}(\Phi) = h_t(\Phi)$.
Let $S_1,\dots,S_n \in \mathcal{H}_\mathfrak{a}$ be a complete set of pairwise non-isomorphic simple objects and set $S \coloneqq S_1 \oplus \cdots \oplus S_n$.
Note that $S$ is a split-generator of $\mathcal{D}$.
Now, by Proposition \ref{t-upper} and Lemma \ref{t-lower}, we get
\begin{equation*}
\delta_t(S,\Phi^N(S)) \leq \hat{\delta}_t^\mathfrak{a}(\Phi^N(S)) \leq \hat{\delta}_t^\mathfrak{a}(S) \delta_t(S,\Phi^N(S))
\end{equation*}
and therefore
\begin{equation*}
\limsup_{N \to \infty} \frac{1}{N} \log \hat{\delta}_t^\mathfrak{a}(\Phi^N(S)) \leq \lim_{N \to \infty} \frac{1}{N} \log \delta_t(S,\Phi^N(S)) = h_t(\Phi) \leq \liminf_{N \to \infty} \frac{1}{N} \log \hat{\delta}_t^\mathfrak{a}(\Phi^N(S)).
\end{equation*}
This shows that the limit of the sequence $\left\{ \frac{1}{N} \log \hat{\delta}_t^\mathfrak{a}(\Phi^N(S)) \right\}_{N=1}^\infty$ exists and coincides with $h_t(\Phi)$.
We then conclude that
\begin{equation*}
\hat{h}_t^\mathfrak{a}(\Phi) = \lim_{N \to \infty} \frac{1}{N} \log \hat{\delta}_t^\mathfrak{a}(\Phi^N(S)) = h_t(\Phi)
\end{equation*}
by (Step 1).

(Step 3)
Finally, let us show that the limit of the sequence $\left\{ \frac{1}{N} \log \hat{\delta}_t^\mathfrak{a}(\Phi^N(G)) \right\}_{N=1}^\infty$ exists for any split-generator $G$ of $\mathcal{D}$.
Let $S \in \mathcal{D}$ be as in (Step 2) and fix a split-generator $G$ of $\mathcal{D}$.
Then, by Lemma \ref{dhkk-lem} (3) and Proposition \ref{t-upper}, we have
\begin{equation*}
\hat{\delta}_t^\mathfrak{a}(\Phi^N(S)) \leq \hat{\delta}_t^\mathfrak{a}(\Phi^N(G)) \delta_t(\Phi^N(G),\Phi^N(S)) \leq \hat{\delta}_t^\mathfrak{a}(\Phi^N(G)) \delta_t(G,S)
\end{equation*}
and similarly
\begin{equation*}
\hat{\delta}_t^\mathfrak{a}(\Phi^N(G)) \leq \hat{\delta}_t^\mathfrak{a}(\Phi^N(S)) \delta_t(S,G).
\end{equation*}
Since the limit of the sequence $\left\{ \frac{1}{N} \log \hat{\delta}_t^\mathfrak{a}(\Phi^N(S)) \right\}_{N=1}^\infty$ exists by (Step 2), we see that
\begin{equation*}
\limsup_{N \to \infty} \frac{1}{N} \log \hat{\delta}_t^\mathfrak{a}(\Phi^N(G)) \leq \lim_{N \to \infty} \frac{1}{N} \log \hat{\delta}_t^\mathfrak{a}(\Phi^N(S)) \leq \liminf_{N \to \infty} \frac{1}{N} \log \hat{\delta}_t^\mathfrak{a}(\Phi^N(G)).
\end{equation*}
Therefore the limit of the sequence $\left\{ \frac{1}{N} \log \hat{\delta}_t^\mathfrak{a}(\Phi^N(G)) \right\}_{N=1}^\infty$ also exists.
\end{proof}

\subsection{Entropy as dimension growth via t-structures}

In this section, we give a way to estimate or compute the categorical entropy using a linear algebraic quantity under the presence of a bounded t-structure with finite heart.

Let $\mathcal{D}$ be a $\mathbf{K}$-linear idempotent complete triangulated category.
Denote by $D^b(\mathbf{K})$ the bounded derived category of $\mathbf{K}$-vector spaces.
The following theorem allows us to interpret the categorical entropy as the exponential growth rate of the dimensions of a sequence of some vector spaces.
In this sense, this theorem can be thought of as an analogue of Theorem \ref{dhkk-thm}.

\begin{thm}\label{t-thm-2}
Let $\mathfrak{a}$ be a bounded t-structure on $\mathcal{D}$ with finite heart and $\Gamma : \mathcal{D} \to D^b(\mathbf{K})$ (resp. $\Gamma : \mathcal{D}^\mathrm{op} \to D^b(\mathbf{K})$) be an exact functor.
For an exact endofunctor $\Phi : \mathcal{D} \to \mathcal{D}$, we have
\begin{gather*}
\hat{h}_t^\mathfrak{a}(\Phi) \geq \limsup_{N \to \infty} \frac{1}{N} \log \sum_{k \in \mathbf{Z}} \dim H^{-k}(\Gamma(\Phi^N(E))) e^{kt}\\
\left( \text{resp. } \hat{h}_t^\mathfrak{a}(\Phi) \geq \limsup_{N \to \infty} \frac{1}{N} \log \sum_{k \in \mathbf{Z}} \dim H^{-k}(\Gamma(\Phi^N(E))) e^{-kt} \right)
\end{gather*}
for any object $E \in \mathcal{D}$.
Moreover we have
\begin{gather*}
\hat{h}_t^\mathfrak{a}(\Phi) = \lim_{N \to \infty} \frac{1}{N} \log \sum_{k \in \mathbf{Z}} \dim H^{-k}(\Gamma(\Phi^N(G))) e^{kt}\\
\left( \text{resp. } \hat{h}_t^\mathfrak{a}(\Phi) = \lim_{N \to \infty} \frac{1}{N} \log \sum_{k \in \mathbf{Z}} \dim H^{-k}(\Gamma(\Phi^N(G))) e^{-kt} \right)
\end{gather*}
for any split-generator $G$ of $\mathcal{D}$ provided that $\Gamma(E) \cong \mathbf{K}$ (concentrated in degree 0) for every simple object $E \in \mathcal{H}_\mathfrak{a}$.
\end{thm}

\begin{rmk}
The latter part of Theorem \ref{t-thm-2} still holds under a mild assumption that $H^k(\Gamma(E)) \cong 0$ for every simple object $E \in \mathcal{H}_\mathfrak{a}$ and $k \neq 0$.
However, for the sake of simplicity, we will keep assuming that $\Gamma(E) \cong \mathbf{K}$ for every simple object $E \in \mathcal{H}_\mathfrak{a}$.
\end{rmk}

\begin{lem}[{\cite[Lemma 2.4]{DHKK}}]\label{lin-lem}
For $V \in D^b(\mathbf{K})$, we have
\begin{equation*}
\delta_t^{D^b(\mathbf{K})}(\mathbf{K},V) = \sum_{k \in \mathbf{Z}} \dim H^{-k}(V) e^{kt}.
\end{equation*}
\end{lem}

\begin{proof}[Proof of Theorem \ref{t-thm-2}]
Note that it is enough to prove the theorem for an exact functor $\Gamma : \mathcal{D} \to D^b(\mathbf{K})$.
Indeed if we have an exact functor $\Gamma : \mathcal{D}^\mathrm{op} \to D^b(\mathbf{K})$ then, by Lemma \ref{op-ent} and Theorem \ref{t-thm-1}, we have
\begin{equation*}
\hat{h}_t^{\mathcal{D},\mathfrak{a}}(\Phi) = h_t^\mathcal{D}(\Phi) = h_{-t}^{\mathcal{D}^\mathrm{op}}(\Phi^\mathrm{op}) = \hat{h}_{-t}^{\mathcal{D}^\mathrm{op},\mathfrak{a}^\mathrm{op}}(\Phi^\mathrm{op})
\end{equation*}
where $\mathfrak{a}^\mathrm{op} \coloneqq (\mathcal{D}^{\geq 0},\mathcal{D}^{\leq 0})$ is the opposite t-structure of $\mathfrak{a} = (\mathcal{D}^{\leq 0},\mathcal{D}^{\geq 0})$.

Let us first prove the lower bound for $\hat{h}_t^\mathfrak{a}(\Phi)$.
If $\Gamma : \mathcal{D} \to D^b(\mathbf{K})$ is zero then it is obvious.
Assume from now on that $\Gamma : \mathcal{D} \to D^b(\mathbf{K})$ is non-zero.
Let $S_1,\dots,S_n \in \mathcal{H}_\mathfrak{a}$ be a complete set of pairwise non-isomorphic simple objects and set $S \coloneqq S_1 \oplus \cdots \oplus S_n$.
Then, for any object $E \in \mathcal{D}$, we have
\begin{equation*}
\hat{\delta}_t^\mathfrak{a}(\Phi^N(E)) \geq \delta_t(S,\Phi^N(E)) \geq \delta_t(\Gamma(S),\Gamma(\Phi^N(E))) \geq \frac{1}{\delta_t(\mathbf{K},\Gamma(S))} \delta_t(\mathbf{K},\Gamma(\Phi^N(E)))
\end{equation*}
by Lemmas \ref{dhkk-lem} (1), (3) and \ref{t-lower}.
Note that $\delta_t(\mathbf{K},\Gamma(S)) \neq 0$ by Lemma \ref{lin-lem} since $\Gamma(S) \not\cong 0$.
Then, by Theorem \ref{t-thm-1} and Lemma \ref{lin-lem}, we obtain
\begin{align*}
\hat{h}_t^\mathfrak{a}(\Phi)
&\geq \limsup_{N \to \infty} \frac{1}{N} \log \hat{\delta}_t^\mathfrak{a}(\Phi^N(E))\\
&\geq \limsup_{N \to \infty} \frac{1}{N} \log \delta_t(\mathbf{K},\Gamma(\Phi^N(E)))\\
&= \limsup_{N \to \infty} \frac{1}{N} \log \sum_{k \in \mathbf{Z}} \dim H^{-k}(\Gamma(\Phi^N(E))) e^{kt}.
\end{align*}

Suppose we have an exact functor $\Gamma : \mathcal{D} \to D^b(\mathbf{K})$ such that $\Gamma(E) \cong \mathbf{K}$ for every simple object $E \in \mathcal{H}_\mathfrak{a}$.
Let $G$ be a split-generator of $\mathcal{D}$.
As in the proof of Lemma \ref{t-lower}, we have a filtration
\begin{equation}\label{s-filt}
\begin{tikzcd}[column sep=tiny]
0 \ar[rr] & & * \ar[dl] & \cdots & * \ar[rr] & & * \ar[dl] & \cdots & * \ar[rr] & & * \ar[dl] & \cdots & * \ar[rr] &  & \Phi^N(G) \ar[dl]\\
& \Sigma^{k_1}S_1^{(1)} \ar[ul,dashed] & & & & \Sigma^{k_1}S_1^{(\lambda_1)} \ar[ul,dashed] & & & & \Sigma^{k_m}S_m^{(1)} \ar[ul,dashed] & & & & \Sigma^{k_m}S_m^{(\lambda_m)} \ar[ul,dashed] &
\end{tikzcd}
\end{equation}
with $k_1>\cdots>k_m$ where $S_i^{(j)} \in \{S_1,\dots,S_n\}$ are the composition factors of $H_\mathfrak{a}^{-k_i}(\Phi^N(G))$ and $\lambda_i \coloneqq \lambda_{\mathcal{H}_\mathfrak{a}}(H_\mathfrak{a}^{-k_i}(\Phi^N(G)))$.
Then, applying $\Gamma$ to the filtration \eqref{s-filt}, we obtain
\begin{equation*}
\begin{tikzcd}[column sep=tiny]
0 \ar[rr] & & * \ar[dl] & \cdots & * \ar[rr] & & * \ar[dl] & \cdots & * \ar[rr] & & * \ar[dl] & \cdots & * \ar[rr] &  & \Gamma(\Phi^N(G)) \ar[dl]\\
& \Sigma^{k_1}\mathbf{K} \ar[ul,dashed] & & & & \Sigma^{k_1}\mathbf{K} \ar[ul,dashed] & & & & \Sigma^{k_m}\mathbf{K} \ar[ul,dashed] & & & & \Sigma^{k_m}\mathbf{K} \ar[ul,dashed] &
\end{tikzcd}.
\end{equation*}
Moreover, since $k_1 > \cdots > k_m$, we see that
\begin{equation*}
H^{-k_i}(\Gamma(\Phi^N(G))) \cong \mathbf{K}^{\lambda_i}
\end{equation*}
for every $1 \leq i \leq m$ and $H^{-k}(\Gamma(\Phi^N(G))) \cong 0$ for every $k \not\in \{k_1,\dots,k_m\}$.
Consequently, we have
\begin{align*}
\hat{h}_t^\mathfrak{a}(\Phi)
&= \lim_{N \to \infty} \frac{1}{N} \log \sum_{i=1}^m \lambda_{\mathcal{H}_\mathfrak{a}}(H_\mathfrak{a}^{-k_i}(\Phi^N(G))) e^{k_it}\\
&= \lim_{N \to \infty} \frac{1}{N} \log \sum_{i=1}^m \lambda_i e^{k_it}\\
&= \lim_{N \to \infty} \frac{1}{N} \log \sum_{k \in \mathbf{Z}} \dim H^{-k}(\Gamma(\Phi^N(G))) e^{kt}
\end{align*}
as desired.
\end{proof}

\section{Entropy and co-t-structures}

\subsection{Co-t-structures}

Let us first review the definition and basic properties of co-t-structures.
The reader may refer to \cite{Jor} for a brief overview of this subject.

Let $\mathcal{D}$ be a $\mathbf{K}$-linear idempotent complete triangulated category.

\begin{dfn}[{\cite[Definition 1.1.1]{Bon},\cite[Definition 2.4]{Pau}}]
A {\em co-t-structure} on $\mathcal{D}$ is a pair $\mathfrak{b} = (\mathcal{D}_{\geq 0},\mathcal{D}_{\leq 0})$ of full additive subcategories of $\mathcal{D}$ closed under taking direct summands satisfying the following conditions:
\begin{enumerate}
\item $\Sigma^{-1}\mathcal{D}_{\geq 0} \subset \mathcal{D}_{\geq 0}$ and $\Sigma\mathcal{D}_{\leq 0} \subset \mathcal{D}_{\leq 0}$.
\item $\mathrm{Hom}(E,F) = 0$ if $E \in \mathcal{D}_{\geq 1},F \in \mathcal{D}_{\leq 0}$ where $\mathcal{D}_{\geq 1} \coloneqq \Sigma^{-1}\mathcal{D}_{\geq 0}$.
\item For every $E \in \mathcal{D}$, there exists an exact triangle $D \to E \to F \to \Sigma D$ with $D \in \mathcal{D}_{\geq 1},F \in \mathcal{D}_{\leq 0}$.
\end{enumerate}
If they further satisfy
\begin{equation*}
\mathcal{D} = \bigcup_{k \in \mathbf{Z}} \Sigma^k\mathcal{D}_{\geq 0} = \bigcup_{k \in \mathbf{Z}} \Sigma^k\mathcal{D}_{\leq 0}
\end{equation*}
then $\mathfrak{b}$ is called {\em bounded}.
\end{dfn}

Let $\mathcal{A}$ be an additive category.
An object of $\mathcal{A}$ is called {\em indecomposable} if it is not isomorphic to a direct sum of two non-zero objects.
An additive category is called a {\em Krull--Schmidt category} if every its object is isomorphic to a finite direct sum of objects whose endomorphism rings are local.
It is known that an object of a Krull--Schmidt category is indecomposable if and only if its endomorphism ring is local.
Moreover the number and the isomorphism classes of the indecomposable direct summands in a direct sum decomposition of an object does not depend on the choice of a direct sum decomposition.

\begin{dfn}
Let $\mathfrak{b} = (\mathcal{D}_{\geq 0},\mathcal{D}_{\leq 0})$ be a co-t-structure on $\mathcal{D}$.
The subcategory $\mathcal{C}_\mathfrak{b} \coloneqq \mathcal{D}_{\geq 0} \cap \mathcal{D}_{\leq 0}$, called the {\em co-heart} of $\mathfrak{b}$, is an additive category.
We say $\mathcal{C}_\mathfrak{b}$ is {\em finite} if it is a Krull--Schmidt category and only has finitely many isomorphism classes of indecomposable objects.
If the co-heart $\mathcal{C}_\mathfrak{b}$ is finite then, for $E \in \mathcal{C}_\mathfrak{b}$, we define $\sigma_{\mathcal{C}_\mathfrak{b}}(E) \geq 0$ as the number of indecomposable direct summands in a direct sum decomposition of $E$ in $\mathcal{C}_\mathfrak{b}$.
\end{dfn}

The following is an analogue of Proposition \ref{t-filt}.

\begin{prop}[{\cite[Proposition 1.5.6]{Bon}}]\label{c-filt}
Let $\mathfrak{b}$ be a bounded co-t-structure on $\mathcal{D}$.
Then, for any $E \in \mathcal{D}$, there exists a filtration
\begin{equation*}
\begin{tikzcd}[column sep=tiny]
0 \ar[rr] & & * \ar[dl] \ar[rr] & & * \ar[dl] & \cdots & * \ar[rr] & & E \ar[dl]\\
& \Sigma^{k_1}E_1 \ar[ul,dashed] & & \Sigma^{k_2}E_2 \ar[ul,dashed] & & & & \Sigma^{k_m}E_m \ar[ul,dashed] &
\end{tikzcd}
\end{equation*}
such that $E_i \in \mathcal{C}_\mathfrak{b}$ and $k_1<\cdots<k_m$.
\end{prop}

\begin{rmk}\label{t-vs-c}
In this paper, a filtration as in Proposition \ref{c-filt} will be called a {\em co-t-filtration} of $E$ with respect to $\mathfrak{b}$.
In contrast with a t-filtration, a co-t-filtration is not unique.
For example, the zero object has a co-t-filtration of the form
\begin{equation*}
\begin{tikzcd}[column sep=tiny]
0 \ar[rr] & & \Sigma^{-1}E \ar[dl,"\mathrm{id}"] \ar[rr] & & 0 \ar[dl]\\
& \Sigma^{-1}E \ar[ul,dashed] & & E \ar[ul,dashed,"\mathrm{id}"] &
\end{tikzcd}
\end{equation*}
for every $E \in \mathcal{C}_\mathfrak{b}$.
\end{rmk}

\subsection{Entropy via co-t-structures}

The goal of this section is to define an entropy-type invariant using a bounded co-t-structure with finite co-heart and study its relationship with the categorical entropy.

Let $\mathcal{D}$ be a $\mathbf{K}$-linear idempotent complete triangulated category.

\begin{dfn}
Let $\mathfrak{b}$ be a bounded co-t-structure on $\mathcal{D}$ with finite co-heart.
Then, for $E \in \mathcal{D}$ and $t \in \mathbf{R}$, we define
\begin{equation*}
\check{\delta}_t^{\mathcal{D},\mathfrak{b}}(E) \coloneqq \inf
\left\{\sum_{i=1}^m \sigma_{\mathcal{C}_\mathfrak{b}}(E_i) e^{k_i t} \,\left|\,
\begin{tikzcd}[column sep=tiny]
0 \ar[rr] & & * \ar[dl] \ar[rr] & & * \ar[dl] & \cdots & * \ar[rr] & & E \ar[dl]\\
& \Sigma^{k_1}E_1 \ar[ul,dashed] & & \Sigma^{k_2}E_2 \ar[ul,dashed] & & & & \Sigma^{k_m}E_m \ar[ul,dashed] &
\end{tikzcd}
\text{: a co-t-filtration}
\right.\right\}.
\end{equation*}
We simply write $\check{\delta}_t^\mathfrak{b}(E) \coloneqq \check{\delta}_t^{\mathcal{D},\mathfrak{b}}(E)$ if $\mathcal{D}$ is clear from the context.
\end{dfn}

\begin{dfn}\label{c-ent}
Let $\mathfrak{b}$ be a bounded co-t-structure on $\mathcal{D}$ with finite co-heart and $\Phi : \mathcal{D} \to \mathcal{D}$ be an exact endofunctor.
Then we define
\begin{equation*}
\check{h}_t^{\mathcal{D},\mathfrak{b}}(\Phi) \coloneqq \sup_{E \in \mathcal{D}} \limsup_{N \to \infty} \frac{1}{N} \log \check{\delta}_t^{\mathcal{D},\mathfrak{b}}(\Phi^N(E)).
\end{equation*}
We simply write $\check{h}_t^\mathfrak{b}(\Phi) \coloneqq \check{h}_t^{\mathcal{D},\mathfrak{b}}(\Phi)$ if $\mathcal{D}$ is clear from the context.
\end{dfn}

The following theorem corresponds to Theorem \ref{t-thm-1}.
Like Theorem \ref{t-thm-1}, it states that $\check{h}_t^\mathfrak{b}$ also coincides with the categorical entropy.

\begin{thm}\label{c-thm-1}
Let $\mathfrak{b}$ be a bounded co-t-structure on $\mathcal{D}$ with finite co-heart and $\Phi : \mathcal{D} \to \mathcal{D}$ be an exact endofunctor.
Then we have
\begin{equation*}
h_t(\Phi) = \check{h}_t^\mathfrak{b}(\Phi) = \lim_{N \to \infty} \frac{1}{N} \log \check{\delta}_t^\mathfrak{b}(\Phi^N(G))
\end{equation*}
for any split-generator $G$ of $\mathcal{D}$.
In particular, $\check{h}_t^\mathfrak{b}(\Phi)$ does not depend on the choice of $\mathfrak{b}$.
\end{thm}

The strategy of the proof of Theorem \ref{c-thm-1} is the same as that of Theorem \ref{t-thm-1}.
We need the following proposition which corresponds to Proposition \ref{t-ex-prop}.

\begin{prop}\label{c-ex-prop}
Let $\mathfrak{b}$ be a bounded co-t-structure on $\mathcal{D}$ with finite co-heart.
If there is an exact triangle $D \to E \to F \to \Sigma D$ in $\mathcal{D}$ then
\begin{equation*}
\check{\delta}_t^\mathfrak{b}(E) \leq \check{\delta}_t^\mathfrak{b}(D) + \check{\delta}_t^\mathfrak{b}(F).
\end{equation*}
\end{prop}

\begin{proof}
Fix $\varepsilon>0$.
By the definitions of $\check{\delta}_t^\mathfrak{b}(D)$ and $\check{\delta}_t^\mathfrak{b}(F)$, there are co-t-filtrations
\begin{equation*}
\begin{tikzcd}[column sep=tiny]
0 \ar[rr] & & * \ar[dl] \ar[rr] & & * \ar[dl] & \cdots & * \ar[rr] & & D \ar[dl]\\
& \Sigma^{k_1}D_1 \ar[ul,dashed] & & \Sigma^{k_2}D_2 \ar[ul,dashed] & & & & \Sigma^{k_m}D_m \ar[ul,dashed] &
\end{tikzcd}
\end{equation*}
and
\begin{equation*}
\begin{tikzcd}[column sep=tiny]
0 \ar[rr] & & * \ar[dl] \ar[rr] & & * \ar[dl] & \cdots & * \ar[rr] & & F \ar[dl]\\
& \Sigma^{k'_1}F_1 \ar[ul,dashed] & & \Sigma^{k'_2}F_2 \ar[ul,dashed] & & & & \Sigma^{k'_{m'}}F_{m'} \ar[ul,dashed] &
\end{tikzcd}
\end{equation*}
such that
\begin{equation*}
\sum_{i=1}^m \sigma_{\mathcal{C}_\mathfrak{b}}(D_i) e^{k_i t} < \check{\delta}_t^\mathfrak{b}(D) + \varepsilon,\quad \sum_{i=1}^{m'} \gamma_{\mathcal{C}_\mathfrak{b}}(F_i) e^{k'_i t} < \check{\delta}_t^\mathfrak{b}(F) + \varepsilon.
\end{equation*}
From these filtrations, as in the proof of Proposition \ref{t-ex-prop}, we get a filtration
\begin{equation}\label{not-c-filt}
\begin{tikzcd}[column sep=tiny]
0 \ar[rr] & & * \ar[dl] & \cdots & * \ar[rr] & & * \ar[dl] \ar[rr] & & * \ar[dl] & \cdots & * \ar[rr] &  & E \ar[dl]\\
& \Sigma^{k''_1}E_1 \ar[ul,dashed] & & & & \Sigma^{k''_m}E_m \ar[ul,dashed] & & \Sigma^{k''_{m+1}}E_{m+1} \ar[ul,dashed] & & & & \Sigma^{k''_{m''}}E_{m''} \ar[ul,dashed] &
\end{tikzcd}
\end{equation}
where $E_1 \coloneqq D_1,\dots,E_m \coloneqq D_m,E_{m+1} \coloneqq F_1,\dots,E_{m''} \coloneqq F_{m'}$ and $k''_1 \coloneqq k_1,\dots,k''_m \coloneqq k_m,k''_{m+1} \coloneqq k_1',\dots,k''_{m''} \coloneqq k'_{m'}$.

In general, the filtration \eqref{not-c-filt} is not a co-t-filtration, i.e., there may be some $1 \leq i < m''$ such that $k''_i \geq k''_{i+1}$.
However, in that case, we can exchange $\Sigma^{k''_i}E_i$ and $\Sigma^{k''_{i+1}}E_{i+1}$ without changing the other $\Sigma^{k''_j}E_j$'s.
To see this, consider the following part of the filtration \eqref{not-c-filt}
\begin{equation}\label{not-c-part}
\begin{tikzcd}[column sep=tiny]
A \ar[rr] & & B \ar[dl,"g"] \ar[rr] & & C \ar[dl]\\
& \Sigma^{k''_i}E_i \ar[ul,dashed] & & \Sigma^{k''_{i+1}}E_{i+1} \ar[ul,dashed,"f"] &
\end{tikzcd}.
\end{equation}
Then $\Sigma g \circ f \in \mathrm{Hom}(\Sigma^{k''_{i+1}}E_{i+1},\Sigma^{k''_i+1}E_i) = \mathrm{Hom}(E_{i+1},\Sigma^{k''_i-k''_{i+1}+1}E_i) = 0$ since $E_i,E_{i+1} \in \mathcal{C}_\mathfrak{b}$ and $k''_i-k''_{i+1}+1>0$.
Therefore, by Lemma \ref{ex-lem-2}, we obtain an exact triangle $A \to C \to \Sigma^{k''_i}E_i \oplus \Sigma^{k''_{i+1}}E_{i+1} \to \Sigma A$ and so we can replace \eqref{not-c-part} by
\begin{equation*}
\begin{tikzcd}[column sep=tiny]
A \ar[rr] & & C \ar[dl]\\
& \Sigma^{k''_i}(E_i \oplus E_{i+1}) \ar[ul,dashed] &
\end{tikzcd}
\end{equation*}
if $k''_i=k''_{i+1}$ or by
\begin{equation*}
\begin{tikzcd}[column sep=tiny]
A \ar[rr] & & * \ar[dl] \ar[rr] & & C \ar[dl]\\
& \Sigma^{k''_{i+1}}E_{i+1} \ar[ul,dashed] & & \Sigma^{k''_i}E_i \ar[ul,dashed] &
\end{tikzcd}
\end{equation*}
if $k''_i>k''_{i+1}$ using Lemma \ref{ex-lem-1}.

The above argument shows that $E$ has a co-t-filtration whose factors are $\Sigma^{k'''_1}E'_1,\dots,\Sigma^{k'''_{m'''}}E'_{m'''}$ such that
\begin{equation*}
\sum_{i=1}^{m'''} \sigma_{\mathcal{C}_\mathfrak{b}}(E'_i) e^{k'''_i t} = \sum_{i=1}^{m''} \sigma_{\mathcal{C}_\mathfrak{b}}(E_i) e^{k''_i t}.
\end{equation*}
Thus, by the definition of $\check{\delta}_t^\mathfrak{b}(E)$, we obtain
\begin{equation*}
\check{\delta}_t^\mathfrak{b}(E) \leq \sum_{i=1}^{m'''} \sigma_{\mathcal{C}_\mathfrak{b}}(E'_i) e^{k'''_i t} = \sum_{i=1}^{m''} \sigma_{\mathcal{C}_\mathfrak{b}}(E_i) e^{k''_i t} =\sum_{i=1}^m \sigma_{\mathcal{C}_\mathfrak{b}}(D_i) e^{k_i t} + \sum_{i=1}^{m'} \sigma_{\mathcal{C}_\mathfrak{b}}(F_i) e^{k'_i t} < \check{\delta}_t^\mathfrak{b}(D) + \check{\delta}_t^\mathfrak{b}(F) + 2\varepsilon.
\end{equation*}
The assertion then follows since $\varepsilon>0$ is arbitrary.
\end{proof}

The following is an analogue of Proposition \ref{t-upper}.

\begin{prop}\label{c-upper}
Let $\mathfrak{b}$ be a bounded co-t-structure on $\mathcal{D}$ with finite co-heart.
Then
\begin{equation*}
\check{\delta}_t^\mathfrak{b}(E) \leq \check{\delta}_t^\mathfrak{b}(D) \delta_t(D,E)
\end{equation*}
for any objects $D,E \in \mathcal{D}$.
\end{prop}

\begin{proof}
This can be proved as in the proof of Proposition \ref{t-upper} using Proposition \ref{c-ex-prop}.
\end{proof}

The following is an analogue of Lemma \ref{t-lower}.

\begin{lem}\label{c-lower}
Let $\mathfrak{b}$ be a bounded co-t-structure on $\mathcal{D}$ with finite co-heart.
Let $M_1,\dots,M_n \in \mathcal{C}_\mathfrak{b}$ be a complete set of pairwise non-isomorphic indecomposable objects and set $M \coloneqq M_1 \oplus \cdots \oplus M_n$.
Then for any $E \in \mathcal{D}$, we have
\begin{equation*}
\check{\delta}_t^\mathfrak{b}(E) \geq \delta_t(M,E).
\end{equation*}
\end{lem}

\begin{proof}
Fix $\varepsilon>0$.
By the definition of $\check{\delta}_t^\mathfrak{b}(E)$, there is a co-t-filtration
\begin{equation*}
\begin{tikzcd}[column sep=tiny]
0 \ar[rr] & & * \ar[dl] \ar[rr] & & * \ar[dl] & \cdots & * \ar[rr] & & E \ar[dl]\\
& \Sigma^{k_1}E_1 \ar[ul,dashed] & & \Sigma^{k_2}E_2 \ar[ul,dashed] & & & & \Sigma^{k_m}E_m \ar[ul,dashed] &
\end{tikzcd}
\end{equation*}
such that
\begin{equation*}
\sum_{i=1}^m \sigma_{\mathcal{C}_\mathfrak{b}}(E_i) e^{k_i t} < \check{\delta}_t^\mathfrak{b}(E) + \varepsilon.
\end{equation*}
Using direct sum decompositions of $E_1,\dots,E_m \in \mathcal{C}_\mathfrak{b}$, we can produce a filtration
\begin{equation*}
\begin{tikzcd}[column sep=tiny]
0 \ar[rr] & & * \ar[dl] & \cdots & * \ar[rr] & & * \ar[dl] & \cdots & * \ar[rr] & & * \ar[dl] & \cdots & * \ar[rr] &  & E \ar[dl]\\
& \Sigma^{k_1}M_1^{(1)} \ar[ul,dashed] & & & & \Sigma^{k_1}M_1^{(\sigma_{\mathcal{C}_\mathfrak{b}}(E_1))} \ar[ul,dashed] & & & & \Sigma^{k_m}M_m^{(1)} \ar[ul,dashed] & & & & \Sigma^{k_m}M_m^{(\sigma_{\mathcal{C}_\mathfrak{b}}(E_m))} \ar[ul,dashed] &
\end{tikzcd}
\end{equation*}
using Lemma \ref{ex-lem-1} where $M_i^{(j)} \in \{M_1,\dots,M_n\}$ are the indecomposable direct summands of $E_i$.
From this, we obtain a filtration
\begin{equation*}
\begin{tikzcd}[column sep=tiny]
0 \ar[rr] & & * \ar[dl] & \cdots & * \ar[rr] & & * \ar[dl] & \cdots & * \ar[rr] & & * \ar[dl] & \cdots & * \ar[rr] &  & E \oplus M' \ar[dl]\\
& \Sigma^{k_1}M \ar[ul,dashed] & & & & \Sigma^{k_1}M \ar[ul,dashed] & & & & \Sigma^{k_m}M \ar[ul,dashed] & & & & \Sigma^{k_m}M \ar[ul,dashed] &
\end{tikzcd}
\end{equation*}
for some $M' \in \mathcal{D}$ which can be written as a direct sum of some shifts of $M_1,\dots,M_n$.
This shows that
\begin{equation*}
\check{\delta}_t^\mathfrak{b}(E) + \varepsilon > \sum_{i=1}^m \sigma_{\mathcal{C}_\mathfrak{b}}(E_i) e^{k_i t} \geq \delta_t(M,E)
\end{equation*}
where the second inequality follows from the definition of $\delta_t(M,E)$.
Since $\varepsilon>0$ is arbitrary, the lemma follows.
\end{proof}

\begin{proof}[Proof of Theorem \ref{c-thm-1}]
This can be proved as in the proof of Theorem \ref{t-thm-1} using Proposition \ref{c-upper} and Lemma \ref{c-lower}.
\end{proof}

\subsection{Entropy as dimension growth via co-t-structures}

In this section, we give an interpretation of the categorical entropy as the exponential growth rate of the dimensions of a sequence of some vector spaces.

Let $\mathcal{D}$ be a $\mathbf{K}$-linear idempotent complete triangulated category.
The following theorem corresponds to Theorem \ref{t-thm-2}.

\begin{thm}\label{c-thm-2}
Let $\mathfrak{b}$ be a bounded co-t-structure on $\mathcal{D}$ with finite co-heart and $\Gamma : \mathcal{D} \to D^b(\mathbf{K})$ (resp. $\Gamma : \mathcal{D}^\mathrm{op} \to D^b(\mathbf{K})$) be an exact functor.
For an exact endofunctor $\Phi : \mathcal{D} \to \mathcal{D}$, we have
\begin{gather*}
\check{h}_t^\mathfrak{b}(\Phi) \geq \limsup_{N \to \infty} \frac{1}{N} \log \sum_{k \in \mathbf{Z}} \dim H^{-k}(\Gamma(\Phi^N(E))) e^{kt}\\
\left( \text{resp. } \check{h}_t^\mathfrak{b}(\Phi) \geq \limsup_{N \to \infty} \frac{1}{N} \log \sum_{k \in \mathbf{Z}} \dim H^{-k}(\Gamma(\Phi^N(E))) e^{-kt} \right)
\end{gather*}
for any object $E \in \mathcal{D}$.
Moreover we have
\begin{gather*}
\check{h}_t^\mathfrak{b}(\Phi) = \lim_{N \to \infty} \frac{1}{N} \log \sum_{k \in \mathbf{Z}} \dim H^{-k}(\Gamma(\Phi^N(G))) e^{kt}\\
\left( \text{resp. } \check{h}_t^\mathfrak{b}(\Phi) = \lim_{N \to \infty} \frac{1}{N} \log \sum_{k \in \mathbf{Z}} \dim H^{-k}(\Gamma(\Phi^N(G))) e^{-kt} \right)
\end{gather*}
for any split-generator $G$ of $\mathcal{D}$ provided that $\Gamma(E) \cong \mathbf{K}$ (concentrated in degree 0) for every indecomposable object $E \in \mathcal{C}_\mathfrak{b}$ and, for any indecomposable objects $E,F \in \mathcal{C}_\mathfrak{b}$, a morphism $f \in \mathrm{Hom}(E,F)$ is an isomorphism whenever $\Gamma(f) \neq 0$.
\end{thm}

\begin{lem}\label{acyc-lem}
 Suppose we have a diagram
\begin{equation*}
\begin{tikzcd}[column sep=tiny]
A \ar[rr] & & * \ar[dl,"g"] \ar[rr] & & B \ar[dl]\\
& \Sigma^{-1}(F \oplus F') \ar[ul,dashed] & & E \oplus E' \ar[ul,dashed,"f"] &
\end{tikzcd}
\end{equation*}
consisting of exact triangles such that the $\mathrm{Hom}(E,F)$-component of the morphism $\Sigma g \circ f \in \mathrm{Hom}(E \oplus E',F \oplus F')$ is an isomorphism.
Then there exists a diagram
\begin{equation*}
\begin{tikzcd}[column sep=tiny]
A \ar[rr] & & * \ar[dl] \ar[rr] & & B \ar[dl]\\
& \Sigma^{-1}F' \ar[ul,dashed] & & E'' \ar[ul,dashed] &
\end{tikzcd}
\end{equation*}
consisting of exact triangles such that $E'' \cong E'$.
\end{lem}

\begin{proof}
By Lemma \ref{ex-lem-2}, there is an exact triangle $A \to B \to C \to \Sigma A$ where $C$ fits into the exact triangle $\Sigma^{-1}(F \oplus F') \to C \to E \oplus E' \overset{\Sigma g \circ f}{\longrightarrow} \Sigma F \oplus F'$.
Then, by the octahedral axiom, we have a diagram
\begin{equation*}
\begin{tikzcd}
\Sigma^{-1}F' \ar[r] \ar[d,equal] & \Sigma^{-1}(F \oplus F') \ar[d] \ar[r] & \Sigma^{-1}F \ar[r] \ar[d] & F' \ar[d,equal]\\
\Sigma^{-1}F' \ar[r] & C \ar[r] \ar[d] & E'' \ar[r] \ar[d] & F'\\
& E \oplus E' \ar[r,equal] \ar[d,"\Sigma g \circ f",swap] & E \oplus E' \ar[d,"h"] &\\
& F \oplus F' \ar[r] & F &
\end{tikzcd}
\end{equation*}
where the morphism $\Sigma^{-1}F' \to \Sigma^{-1}(F \oplus F')$ in this diagram is the inclusion.
On the other hand, by assumption, precomposing $h : E \oplus E' \to F$ with the inclusion $E \to E \oplus E'$ is an isomorphism.
Again by the octahedral axiom, we get a diagram
\begin{equation*}
\begin{tikzcd}
& \Sigma^{-1}F \ar[r] \ar[d] & 0\ar[d] &\\
& E'' \ar[r,equal] \ar[d] & E'' \ar[d] &\\
E \ar[r] \ar[d,equal] & E \oplus E' \ar[d,"h"] \ar[r] & E' \ar[r] \ar[d] & \Sigma E \ar[d,equal]\\
E \ar[r] & F \ar[r] & 0 \ar[r] & \Sigma E
\end{tikzcd}
\end{equation*}
and hence $E'' \cong E'$.
We obtain the desired diagram by applying Lemma \ref{ex-lem-1} to two exact triangles $A \to B \to C \to \Sigma A$ and $\Sigma^{-1}F' \to C \to E'' \to F'$.
\end{proof}

\begin{proof}[Proof of Theorem \ref{c-thm-2}]
As in the proof of Theorem \ref{t-thm-2}, it is enough to prove the theorem for an exact functor $\Gamma : \mathcal{D} \to D^b(\mathbf{K})$.
Also the lower bound for $\check{h}_t^\mathfrak{b}(\Phi)$ can be proved as in the proof of Theorem \ref{t-thm-2}.

Suppose we have an exact functor $\Gamma : \mathcal{D} \to D^b(\mathbf{K})$ such that $\Gamma(E) \cong \mathbf{K}$ for every indecomposable object $E \in \mathcal{C}_\mathfrak{b}$ and, for any indecomposable objects $E,F \in \mathcal{C}_\mathfrak{b}$, a morphism $f \in \mathrm{Hom}(E,F)$ is an isomorphism whenever $\Gamma(f) \neq 0$.
Let $M_1,\dots,M_n \in \mathcal{C}_\mathfrak{b}$ be a complete set of pairwise non-isomorphic indecomposable objects.

Let us first see what the assumption on $\Gamma$ implies.
For $E,F \in \mathcal{C}_\mathfrak{b}$, consider a morphism $f \in \mathrm{Hom}(E,F)$.
Fix direct sum decompositions $E \cong M_1^{\oplus k_1} \oplus \cdots \oplus M_n^{\oplus k_n}$ and $F \cong M_1^{\oplus l_1} \oplus \cdots \oplus M_n^{\oplus l_n}$.
If $\Gamma(f) \neq 0$ then it has a $\mathrm{Hom}(\mathbf{K},\mathbf{K})$-component which is non-zero (and thus an isomorphism).
Then the assumption on $\Gamma$ implies that it must come from a $\mathrm{Hom}(M_i,M_i)$-component of $f$ which is an isomorphism for some $1 \leq i \leq n$.
To sum up, this shows that if $\Gamma(f) \neq 0$ then there exist direct sum decompositions $E \cong M_i \oplus E'$ and $F \cong M_i \oplus F'$ for some $1 \leq i \leq n$ and $E',F' \in \mathcal{C}_\mathfrak{b}$ such that the $\mathrm{Hom}(M_i,M_i)$-component of $f$ in these direct sum decompositions is an isomorphism.

Now let $G$ be a split-generator of $\mathcal{D}$.
Take a co-t-filtration
\begin{equation}\label{p-filt}
\begin{tikzcd}[column sep=tiny]
0 \ar[rr] & & * \ar[dl,"g_1"] \ar[rr] & & * \ar[dl,"g_2"] & \cdots & * \ar[rr] & & \Phi^N(G) \ar[dl]\\
& \Sigma^{k_1}E_1 \ar[ul,dashed] & & \Sigma^{k_2}E_2 \ar[ul,dashed,"f_2"] & & & & \Sigma^{k_m}E_m \ar[ul,dashed,"f_m"] &
\end{tikzcd}.
\end{equation}
Then, applying $\Gamma$ to the filtration \eqref{p-filt}, we obtain
\begin{equation}\label{v-filt}
\begin{tikzcd}[column sep=tiny]
0 \ar[rr] & & * \ar[dl] \ar[rr] & & * \ar[dl] & \cdots & * \ar[rr] & & \Gamma(\Phi^N(G)) \ar[dl]\\
& \Sigma^{k_1}\mathbf{K}^{\sigma_1} \ar[ul,dashed] & & \Sigma^{k_2}\mathbf{K}^{\sigma_2} \ar[ul,dashed] & & & & \Sigma^{k_m}\mathbf{K}^{\sigma_m} \ar[ul,dashed] &
\end{tikzcd}
\end{equation}
where $\sigma_i \coloneqq \sigma_{\mathcal{C}_\mathfrak{b}}(E_i)$.

(Case 1)
Suppose $\Gamma(\Sigma g_i \circ f_{i+1}) = 0$ for every $1 \leq i < m$.
Then we see that
\begin{equation*}
H^{-k_i}(\Gamma(\Phi^N(G))) \cong \mathbf{K}^{\sigma_i}
\end{equation*}
for every $1 \leq i \leq m$ and $H^{-k}(\Gamma(\Phi^N(G))) \cong 0$ for every $k \not\in \{k_1,\dots,k_m\}$.
Consequently, we have
\begin{equation*}
\check{\delta}_t^\mathfrak{b}(\Phi^N(G)) \leq \sum_{i=1}^m \sigma_{\mathcal{C}_\mathfrak{b}}(E_i) e^{k_it} = \sum_{i=1}^m \sigma_i e^{k_it} = \sum_{k \in \mathbf{Z}} \dim H^{-k}(\Gamma(\Phi^N(G))) e^{kt}
\end{equation*}
and therefore
\begin{equation*}
\limsup_{N \to \infty} \frac{1}{N} \log \sum_{k \in \mathbf{Z}} \dim H^{-k}(\Gamma(\Phi^N(G))) e^{kt} \leq \check{h}_t^\mathfrak{b}(\Phi) \leq \liminf_{N \to \infty} \frac{1}{N} \log \sum_{k \in \mathbf{Z}} \dim H^{-k}(\Gamma(\Phi^N(G))) e^{kt}
\end{equation*}
where the first inequality follows from the former part of Theorem \ref{c-thm-2}.
This shows that the limit of the sequence $\left\{ \frac{1}{N} \log \sum_{k \in \mathbf{Z}} \dim H^{-k}(\Gamma(\Phi^N(G))) e^{kt} \right\}$ exists and it coincides with $\check{h}_t^\mathfrak{b}(\Phi)$.

(Case 2)
Suppose $\Gamma(\Sigma g_j \circ f_{j+1}) \neq 0$ for some $1 \leq j < m$.
Then we must have $k_j = k_{j+1}-1$.
As mentioned above, the assumption on $\Gamma$ and $\Gamma(\Sigma g_j \circ f_{j+1}) \neq 0$ imply that there exist direct sum decompositions $E_{j+1} \cong M_i \oplus E'_{j+1}$ and $E_j \cong M_i \oplus E'_j$ for some $1 \leq i \leq n$ and $E'_j,E'_{j+1} \in \mathcal{C}_\mathfrak{b}$ such that the $\mathrm{Hom}(\Sigma^{k_{j+1}}M_i,\Sigma^{k_{j+1}}M_i)$-component of $\Sigma g_j \circ f_{j+1}$ in these direct sum decompositions is an isomorphism.
Then, applying Lemma \ref{acyc-lem}, we can replace the following part of the filtration \eqref{v-filt}
\begin{equation*}
\begin{tikzcd}[column sep=tiny]
* \ar[rr] & & * \ar[dl] \ar[rr] & & * \ar[dl]\\
& \Sigma^{k_j}\mathbf{K}^{\sigma_j} \ar[ul,dashed] & & \Sigma^{k_{j+1}}\mathbf{K}^{\sigma_{j+1}} \ar[ul,dashed] &
\end{tikzcd}
\quad\text{by}\quad
\begin{tikzcd}[column sep=tiny]
* \ar[rr] & & * \ar[dl] \ar[rr] & & * \ar[dl]\\
& \Sigma^{k_j}\mathbf{K}^{\sigma_j-1} \ar[ul,dashed] & & \Sigma^{k_{j+1}}\mathbf{K}^{\sigma_{j+1}-1} \ar[ul,dashed] &
\end{tikzcd}
\end{equation*}
and simultaneously the following part of the filtration \eqref{p-filt}
\begin{equation*}
\begin{tikzcd}[column sep=tiny]
* \ar[rr] & & * \ar[dl] \ar[rr] & & * \ar[dl]\\
& \Sigma^{k_j}E_j \ar[ul,dashed] & & \Sigma^{k_{j+1}}E_{j+1} \ar[ul,dashed] &
\end{tikzcd}
\quad\text{by}\quad
\begin{tikzcd}[column sep=tiny]
* \ar[rr] & & * \ar[dl] \ar[rr] & & * \ar[dl]\\
& \Sigma^{k_j}E'_j \ar[ul,dashed] & & \Sigma^{k_{j+1}}E'_{j+1} \ar[ul,dashed] &
\end{tikzcd}.
\end{equation*}
Clearly the resulting filtration is again a co-t-filtration of $\Phi^N(G)$.
Note also that $\sigma_{\mathcal{C}_\mathfrak{b}}(E'_j) = \sigma_j-1$ and $\sigma_{\mathcal{C}_\mathfrak{b}}(E'_{j+1}) = \sigma_{j+1}-1$.
We can repeat this process until the resulting co-t-filtration satisfies the assumption of (Case 1).
\end{proof}

\section{Applications}

\subsection{ST-triples}

In this section, we give an application of the main theorems of this paper, in particular, Theorems \ref{t-thm-2} and \ref{c-thm-2}.
It is also one of the motivations of this paper.

Let $\mathcal{D}$ be a $\mathbf{K}$-linear idempotent complete algebraic triangulated category.
An object $M \in \mathcal{D}$ is called a {\em silting object} if it is a split-generator of $\mathcal{D}$ and satisfies $\mathrm{Hom}(M,\Sigma^k M) = 0$ for every $k>0$.
For a full subcategory $\mathcal{C} \subset \mathcal{D}$ and an object $M \in \mathcal{D}$, we define
\begin{gather*}
\mathcal{C}_M^{\leq 0} \coloneqq \{ E \in \mathcal{C} \,|\, \mathrm{Hom}(M,\Sigma^kE) = 0 \text{ for all } k > 0 \},\\
\mathcal{C}_M^{\geq 0} \coloneqq \{ E \in \mathcal{C} \,|\, \mathrm{Hom}(M,\Sigma^kE) = 0 \text{ for all } k < 0 \}.
\end{gather*}

\begin{dfn}[{\cite[Definition 4.3 and Remark 4.5]{AMY}}]\label{st-triple}
An {\em ST-triple} $(\mathcal{B},\mathcal{A},M)$ inside $\mathcal{D}$ consists of thick subcategories $\mathcal{A},\mathcal{B}$ of $\mathcal{D}$ and a silting object $M \in \mathcal{B}$ satisfying the following conditions:
\begin{enumerate}
\item $\mathrm{Hom}(M,E)$ is finite-dimensional for every $E \in \mathcal{D}$.
\item $\bar{\mathfrak{a}} \coloneqq (\mathcal{D}_M^{\leq 0},\mathcal{D}_M^{\geq 0})$ is a t-structure on $\mathcal{D}$.
\item $\mathcal{D}_M^{\geq 0} \subset \mathcal{A}$ and $\mathfrak{a} \coloneqq (\mathcal{A}_M^{\leq 0},\mathcal{A}_M^{\geq 0})$ is a bounded t-structure on $\mathcal{A}$.
\end{enumerate}
\end{dfn}

\begin{rmk}
When $(\mathcal{B},\mathcal{A},M)$ is an ST-triple inside $\mathcal{D}$, the pair $(\mathcal{B},\mathcal{A})$ is called an {\em ST-pair} inside $\mathcal{D}$.
It is known that if $(\mathcal{B},\mathcal{A})$ is an ST-pair then $(\mathcal{B},\mathcal{A},M)$ is an ST-triple for any silting object $M \in \mathcal{B}$ (see \cite[Proposition 5.2]{AMY}).
\end{rmk}

Let $A$ be a dg $\mathbf{K}$-algebra.
We write its derived category by $D(A)$.
The {\em perfect derived category} $\mathrm{per}(A)$ is defined to be $\mathrm{thick}(A)$ and the {\em finite-dimensional derived category} $D_\mathrm{fd}(A)$ is defined as the full subcategory of $D(A)$ consisting of dg $A$-modules $M$ such that $H^*(M)$ is finite-dimensional.
We also define $D_\mathrm{fd}^-(A)$ as the full subcategory of $D(A)$ consisting of dg $A$-modules $M$ such that $H^k(M)$ is finite-dimensional for every $k \in \mathbf{Z}$ and $H^k(M)=0$ for all $k \gg 0$.

\begin{exa}\label{st-ex}
(1) (\cite[Proposition 6.12]{AMY})
Let $A$ be a dg $\mathbf{K}$-algebra satisfying the following conditions:
\begin{itemize}
\item $H^k(A)=0$ for every $k>0$.
\item $H^k(A)$ is finite-dimensional for every $k \in \mathbf{Z}$.
\end{itemize}
Then $(\mathrm{per}(A),D_\mathrm{fd}(A),A)$ is an ST-triple inside $D_\mathrm{fd}^-(A)$.
Note that the t-structure $\bar{\mathfrak{a}} = (D_\mathrm{fd}^-(A)_A^{\leq 0},D_\mathrm{fd}^-(A)_A^{\geq 0})$ on $D_\mathrm{fd}^-(A)$ in Definition \ref{st-triple} (2) is the standard t-structure
\begin{gather*}
D_\mathrm{fd}^-(A)_A^{\leq 0} = \{ M \in D_\mathrm{fd}^-(A) \,|\, H^k(M) = 0 \text{ for all } k > 0 \},\\
D_\mathrm{fd}^-(A)_A^{\geq 0} = \{ M \in D_\mathrm{fd}^-(A) \,|\, H^k(M) = 0 \text{ for all } k < 0 \}.
\end{gather*}
Also note that $\mathrm{per}(A)$ and $D_\mathrm{fd}(A)$ need not be comparable unlike the following two special cases.

(2) (\cite[Lemma 4.15]{AMY}, \cite[Section 2]{Ami}, \cite[Proposition 2.1]{KY1})
Let $A$ be a dg $\mathbf{K}$-algebra satisfying the following conditions:
\begin{itemize}
\item $H^k(A)=0$ for every $k>0$.
\item $H^0(A)$ is finite-dimensional.
\item $D_\mathrm{fd}(A) \subset \mathrm{per}(A)$.
\end{itemize}
Then $(\mathrm{per}(A),D_\mathrm{fd}(A),A)$ is an ST-triple inside $\mathrm{per}(A)$.
This example includes the {\em Ginzburg dg algebra} associated to a quiver with potential (see \cite{Gin,KY2}).

Note that the condition $D_\mathrm{fd}(A) \subset \mathrm{per}(A)$ is satisfied if $A$ is {\em smooth} in the sense that $A$, regarded as a dg $A$-$A$-bimodule, is an object of $\mathrm{per}(A^\mathrm{op} \otimes A)$ (see the proof of \cite[Lemma 4.1]{Kel2}).

(3) (\cite[Lemma 4.13]{AMY}, \cite[Proposition 2.1]{KY1})
Let $A$ be a dg $\mathbf{K}$-algebra satisfying the following conditions:
\begin{itemize}
\item $H^k(A)=0$ for every $k>0$.
\item $H^*(A)$ is finite-dimensional.
\end{itemize}
Then $(\mathrm{per}(A),D_\mathrm{fd}(A),A)$ is an ST-triple inside $D_\mathrm{fd}(A)$.
In particular, if $A$ is concentrated in degree 0, we can regard it as a finite-dimensional $\mathbf{K}$-algebra.
In this case, we have equivalences $\mathrm{per}(A) \simeq K^b(\mathrm{proj}\,A)$ and $D_\mathrm{fd}(A) \simeq D^b(\mathrm{mod}\,A)$ where $\mathrm{proj}\,A$ denotes the additive category of projective $A$-modules and $\mathrm{mod}\,A$ denotes the abelian category of $A$-modules.

A dg $\mathbf{K}$-algebra $A$ such that $H^*(A)$ is finite-dimensional is often said to be {\em proper}.
Note that, for a proper dg $\mathbf{K}$-algebra $A$, we have $D_\mathrm{fd}(A) \supset \mathrm{per}(A)$.
\end{exa}

It follows from Definition \ref{st-triple} (1) that $\mathcal{B}$ is Hom-finite.
This shows that $\mathcal{B}$ is Krull--Schmidt as every Hom-finite idempotent complete additive category is Krull--Schmidt (see \cite[Lemma 2.1]{HKP}).
Moreover the hearts of the t-structures in Definition \ref{st-triple} (2), (3) can be described as follows.

\begin{lem}[{\cite[Proposition 4.6]{AMY}}]
Let $(\mathcal{B},\mathcal{A},M)$ be an ST-triple inside $\mathcal{D}$.
Then we have $\mathcal{H}_{\bar{\mathfrak{a}}} = \mathcal{H}_\mathfrak{a}$.
Moreover the functor $\mathrm{Hom}(M,-) : \mathcal{H}_{\bar{\mathfrak{a}}} \to \mathrm{mod}(\mathrm{End}(M))$ is an equivalence.
\end{lem}

Since $\mathrm{End}(M)$ is a finite-dimensional $\mathbf{K}$-algebra, $\mathrm{mod}(\mathrm{End}(M))$ is a length category with finitely many isomorphism classes of simple objects.
Therefore the bounded t-structure $\mathfrak{a}$ in Definition \ref{st-triple} (3) has finite heart.

Let $(\mathcal{B},\mathcal{A},M)$ be an ST-triple inside $\mathcal{D}$.
Let $S_1,\dots,S_n \in \mathcal{H}_\mathfrak{a}$ be a complete set of pairwise non-isomorphic simple objects and $S \coloneqq S_1 \oplus \cdots \oplus S_n$.
For a full subcategory $\mathcal{C} \subset \mathcal{D}$, we define
\begin{gather*}
\mathcal{C}^M_{\geq 0} \coloneqq \{ E \in \mathcal{C} \,|\, \mathrm{Hom}(\Sigma^kE,S) = 0 \text{ for all } k < 0 \},\\
\mathcal{C}^M_{\leq 0} \coloneqq \{ E \in \mathcal{C} \,|\, \mathrm{Hom}(\Sigma^kE,S) = 0 \text{ for all } k > 0 \}.
\end{gather*}

The following proposition gives a dual description of Definition \ref{st-triple}.

\begin{prop}[{\cite[Proposition 4.17 and Remark 4.18]{AMY}}]\label{st-dual}
Let $(\mathcal{B},\mathcal{A},M)$ be an ST-triple inside $\mathcal{D}$.
Let $S_1,\dots,S_n \in \mathcal{H}_\mathfrak{a}$ be a complete set of pairwise non-isomorphic simple objects and $S \coloneqq S_1 \oplus \cdots \oplus S_n$.
Then the following hold:
\begin{enumerate}
\item $\mathrm{Hom}(E,S)$ is finite-dimensional for every $E \in \mathcal{D}$.
\item $\bar{\mathfrak{b}} \coloneqq (\mathcal{D}^M_{\geq 0},\mathcal{D}^M_{\leq 0})$ is a co-t-structure on $\mathcal{D}$.
\item $\mathcal{D}^M_{\geq 0} \subset \mathcal{B}$ and $\mathfrak{b} \coloneqq (\mathcal{B}^M_{\geq 0},\mathcal{B}^M_{\leq 0})$ is a bounded co-t-structure on $\mathcal{B}$.
\end{enumerate}
\end{prop}

Similarly as above, we can deduce from Proposition \ref{st-dual} (1) that $\mathcal{A}$ is Hom-finite and hence Krull--Schmidt.
The co-hearts of the co-t-structures in Proposition \ref{st-dual} (2), (3) can be described as follows.

\begin{lem}[{\cite[Remark 4.18]{AMY}}]
Let $(\mathcal{B},\mathcal{A},M)$ be an ST-triple inside $\mathcal{D}$.
Then we have $\mathcal{C}_{\bar{\mathfrak{b}}} = \mathcal{C}_\mathfrak{b} = \mathrm{add}(M)$.
\end{lem}

It then immediately follows that the bounded co-t-structure $\mathfrak{b}$ in Proposition \ref{st-dual} (3) has finite co-heart.

To sum up, for an ST-triple $(\mathcal{B},\mathcal{A},M)$ inside $\mathcal{D}$, we have both a bounded t-structure $\mathfrak{a}$ on $\mathcal{A}$ with finite heart and a bounded co-t-structure $\mathfrak{b}$ on $\mathcal{B}$ with finite co-heart.
Therefore if we have an exact endofunctor $\Phi : \mathcal{D} \to \mathcal{D}$ which restricts to $\mathcal{A}$ and $\mathcal{B}$ then we can apply Theorems \ref{t-thm-1} and \ref{t-thm-2} to $\Phi|_\mathcal{A} : \mathcal{A} \to \mathcal{A}$ and Theorems \ref{c-thm-1} and \ref{c-thm-2} to $\Phi|_\mathcal{B} : \mathcal{B} \to \mathcal{B}$.

Now let $\mathcal{A}$ be a Krull--Schmidt category.
Consider an object $E \in \mathcal{A}$ which admits a direct sum decomposition $E \cong E_1 \oplus \cdots \oplus E_m$ into indecomposable objects $E_1,\dots,E_m$.
Then $E$ is called {\em basic} if $E_i \not\cong E_j$ for every $i \neq j$.

The following is the main theorem of this section.

\begin{thm}\label{app-thm}
Let $(\mathcal{B},\mathcal{A},M)$ be an ST-triple inside $\mathcal{D}$ such that $M$ is basic and $\Phi : \mathcal{D} \to \mathcal{D}$ be an exact endofunctor.
Assume that $\Phi(\mathcal{A}) \subset \mathcal{A}$ and $\Phi(\mathcal{B}) \subset \mathcal{B}$.
Then we have
\begin{gather*}
h_t^\mathcal{A}(\Phi|_\mathcal{A}) = \lim_{N \to \infty} \frac{1}{N} \log \sum_{k \in \mathbf{Z}} \dim \mathrm{Hom}^{-k}(M,\Phi^N(G_\mathcal{A})) e^{kt},\\
h_t^\mathcal{B}(\Phi|_\mathcal{B}) = \lim_{N \to \infty} \frac{1}{N} \log \sum_{k \in \mathbf{Z}} \dim \mathrm{Hom}^{-k}(\Phi^N(G_\mathcal{B}),S) e^{-kt}
\end{gather*}
for any split-generators $G_\mathcal{A} \in \mathcal{A}$ and $G_\mathcal{B} \in \mathcal{B}$.
In particular, if $\Phi$ is an exact autoequivalence then
\begin{equation*}
h_t^\mathcal{A}(\Phi|_\mathcal{A}) = h_{-t}^\mathcal{B}(\Phi|_\mathcal{B}^{-1}) = \lim_{N \to \infty} \frac{1}{N} \log \sum_{k \in \mathbf{Z}} \dim \mathrm{Hom}^{-k}(M,\Phi^N(S)) e^{kt}.
\end{equation*}
\end{thm}

To prove Theorem \ref{app-thm}, we need the following.

\begin{prop}[{\cite[Corollary 4.8]{AMY}}]\label{amy-prop}
Let $(\mathcal{B},\mathcal{A},M)$ be an ST-triple such that $M$ is basic.
Let $M_1,\dots,M_n \in \mathcal{B}$ be indecomposable objects such that $M \cong M_1 \oplus \cdots \oplus M_n$ and $S_i \in \mathcal{H}_{\bar{\mathfrak{a}}} = \mathcal{H}_\mathfrak{a}$ be the simple top of $H_{\bar{\mathfrak{a}}}^0(M_i)$.
Then we have
\begin{equation*}
\mathrm{Hom}(M_i,\Sigma^kS_j) \cong
\begin{cases}
\mathbf{K} & (i=j \text{ and } k=0),\\
0 & (\text{otherwise}).
\end{cases}
\end{equation*}
\end{prop}

\begin{rmk}\label{alg-cl}
In general, we only have
\begin{equation*}
\mathrm{Hom}(M_i,\Sigma^kS_j) \cong
\begin{cases}
\mathrm{End}(S_i) & (i=j \text{ and } k=0),\\
0 & (\text{otherwise}).
\end{cases}
\end{equation*}
However, since we are assuming that $\mathbf{K}$ is algebraically closed, we see that $\mathrm{End}(S_i) \cong \mathbf{K}$.
\end{rmk}

\begin{proof}[Proof of Theorem \ref{app-thm}]
Let $M_1,\dots,M_n \in \mathcal{B}$ be indecomposable objects such that $M \cong M_1 \oplus \cdots \oplus M_n$ and $S_i \in \mathcal{H}_{\bar{\mathfrak{a}}} = \mathcal{H}_\mathfrak{a}$ be the simple top of $H_{\bar{\mathfrak{a}}}^0(M_i)$.
Note that $S_1,\dots,S_n$ form a complete set of pairwise non-isomorphic simple objects of $\mathcal{H}_\mathfrak{a}$.
Let $S \coloneqq S_1 \oplus \cdots \oplus S_n$.

(Step 1)
By Proposition \ref{amy-prop}, the exact functor $\Gamma_\mathcal{A} \coloneqq \mathrm{RHom}(M,-) : \mathcal{A} \to D^b(\mathbf{K})$ satisfies $\Gamma_\mathcal{A}(E) \cong \mathbf{K}$ for every simple object $E \in \mathcal{H}_\mathfrak{a}$.
Thus we get
\begin{align*}
h_t^\mathcal{A}(\Phi|_\mathcal{A})
&= \hat{h}_t^{\mathcal{A},\mathfrak{a}}(\Phi|_\mathcal{A})\\
&= \lim_{N \to \infty} \frac{1}{N} \log \sum_{k \in \mathbf{Z}} \dim H^{-k}(\Gamma_\mathcal{A}(\Phi^N(G_\mathcal{A}))) e^{kt}\\
&= \lim_{N \to \infty} \frac{1}{N} \log \sum_{k \in \mathbf{Z}} \dim \mathrm{Hom}^{-k}(M,\Phi^N(G_\mathcal{A})) e^{kt}
\end{align*}
for any split-generator $G_\mathcal{A}$ of $\mathcal{A}$ by Theorems \ref{t-thm-1} and \ref{t-thm-2}.

(Step 2)
Let us check that the exact functor $\Gamma_\mathcal{B} \coloneqq \mathrm{RHom}(-,S) : \mathcal{B} \to D^b(\mathbf{K})^\mathrm{op}$ satisfies the conditions in Theorem \ref{c-thm-2}.
First, by Proposition \ref{amy-prop}, we see that $\Gamma_\mathcal{B}(E) \cong \mathbf{K}$ for every indecomposable object $E \in \mathcal{C}_\mathfrak{b}$.
Consider a morphism $f \in \mathrm{Hom}(M_i,M_j)$ for some $1 \leq i,j \leq n$.
If $i \neq j$ then we have $\Gamma_\mathcal{B}(f)=0$ by Proposition \ref{amy-prop}.
Suppose $i=j$.
As $M_i$ is indecomposable, $f$ is either an isomorphism or nilpotent.
Therefore if $f$ is not an isomorphism then it is nilpotent and so $\Gamma_\mathcal{B}(f)$ is also nilpotent.
Then, since $\Gamma_\mathcal{B}(f) \in \mathrm{Hom}(\mathbf{K},\mathbf{K})$, we must have $\Gamma_\mathcal{B}(f)=0$.
This shows that if $\Gamma_\mathcal{B}(f) \neq 0$ then $f$ is an isomorphism.

Now we can apply Theorems \ref{c-thm-1} and \ref{c-thm-2} to get
\begin{align*}
h_t^\mathcal{B}(\Phi|_\mathcal{B})
&= \check{h}_t^{\mathcal{B},\mathfrak{b}}(\Phi|_\mathcal{B})\\
&= \lim_{N \to \infty} \frac{1}{N} \log \sum_{k \in \mathbf{Z}} \dim H^{-k}(\Gamma_\mathcal{B}(\Phi^N(G_\mathcal{B}))) e^{-kt}\\
&= \lim_{N \to \infty} \frac{1}{N} \log \sum_{k \in \mathbf{Z}} \dim \mathrm{Hom}^{-k}(\Phi^N(G_\mathcal{B}),S) e^{-kt}
\end{align*}
for any split-generator $G_\mathcal{B}$ of $\mathcal{B}$.

(Step 3)
If $\Phi$ is an exact autoequivalence then $h_t^\mathcal{A}(\Phi|_\mathcal{A}) = h_{-t}^\mathcal{B}(\Phi|_\mathcal{B}^{-1})$ follows immediately from the former part of Theorem \ref{app-thm} applied to $G_\mathcal{A}=S$ and $G_\mathcal{B}=M$.
\end{proof}

\subsection{ST-triples coming from Fukaya categories}

In this section, we will see some examples of ST-triples coming from symplectic geometry and how Theorem \ref{app-thm} can be interpreted in terms of symplectic geometry.

A {\em Liouville domain} $(X,\theta)$ consists of a compact manifold $X$ with boundary and a one-form $\theta$ on $X$ such that $\omega \coloneqq d\theta$ is symplectic and the vector field $Z$ on $X$ determined by $i_Z\omega=\theta$ points strictly outwards along the boundary $\partial X$.
A submanifold $L$ of a Liouville domain $(X,\theta)$ is called Lagrangian if $\dim L = \frac{1}{2} \dim X$ and $\omega|_L=0$.
A Lagrangian submanifold $L$ is called {\em exact} if $\theta|_L$ is exact.

For a Liouville domain $(X,\theta)$, one can define an $A_\infty$-category $\mathcal{W}(X)$ called the {\em wrapped Fukaya category} of $(X,\theta)$ (see \cite{AS}).
An object of $\mathcal{W}(X)$ is an exact Lagrangian submanifold (with or without boundary) of $(X,\theta)$ satisfying some technical conditions.
For $ L,K \in \mathcal{W}(X)$ intersecting transversely, $\mathrm{Hom}_{\mathcal{W}(X)}(L,K)$ is defined to be the $\mathbf{Z}$-graded $\mathbf{C}$-vector space freely generated by the points of $L \cap K$ equipped with an $A_\infty$-structure given by counting some pseudo-holomorphic disks.
Let us denote by $\mathcal{F}(X)$ the full $A_\infty$-subcategory of $\mathcal{W}(X)$ consisting of exact Lagrangian submanifolds without boundary.
We also denote by $D^\pi\mathcal{W}(X),D^\pi\mathcal{F}(X)$ the idempotent complete derived categories of $\mathcal{W}(X),\mathcal{F}(X)$ respectively.
For more details on Fukaya categories, the reader may refer to \cite{Sei}.

In many cases, the derived wrapped Fukaya category $D^\pi\mathcal{W}(X)$ contains a silting object $L$ such that $\mathrm{End}_{D^\pi\mathcal{W}(X)}(L)$ is finite-dimensional.
If this is the case then there exist a dg $\mathbf{C}$-algebra $A$ and an exact equivalence $\varepsilon : D^\pi\mathcal{W}(X) \to \mathrm{per}(A)$ which takes $L$ to $A$ since $D^\pi\mathcal{W}(X)$ is an algebraic triangulated category by construction (see \cite[Theorem 4.3]{Kel1}).
We expect that, in this case, the exact equivalence $\varepsilon$ takes $D^\pi\mathcal{F}(X)$ to $D_\mathrm{fd}(A)$ and so $(D^\pi\mathcal{W}(X),D^\pi\mathcal{F}(X),L)$ is an ST-triple (compare Example \ref{st-ex} (2)).

Let us give a motivational example.
Let $(X,\theta)$ be a Liouville domain obtained by plumbing unit cotangent bundles $D^*S^2$ of the two-sphere according to the $A_n$ or $D_n$ tree.
Then the derived wrapped Fukaya category $D^\pi\mathcal{W}(X)$ contains a silting object $L$ which can be written as the direct sum of the fibers of $D^*S^2$'s.
Moreover it is known that the $A_\infty$-algebra $\mathrm{End}_{\mathcal{W}(X)}(L)$ is quasi-isomorphic to the 2-Calabi--Yau Ginzburg dg algebra $A$ associated to the same tree (see \cite[Theorem 1]{EL}).
Also it can be shown that the exact equivalence $\varepsilon : D^\pi\mathcal{W}(X) \to \mathrm{per}(A)$ takes $D^\pi\mathcal{F}(X)$ to $D_\mathrm{fd}(A)$.
Therefore we see that $(D^\pi\mathcal{W}(X),D^\pi\mathcal{F}(X),L)$ is an ST-triple.
Note that, in this case, the simple objects of $\mathcal{H}_\mathfrak{a}$ correspond to the zero sections of $D^*S^2$'s.

Now suppose $(D^\pi\mathcal{W}(X),D^\pi\mathcal{F}(X),L)$ is an ST-triple.
For an exact autoequivalence $\Phi : D^\pi\mathcal{W}(X) \to D^\pi\mathcal{W}(X)$ such that $\Phi(D^\pi\mathcal{F}(X)) \subset D^\pi\mathcal{F}(X)$, we have
\begin{equation}\label{fuk-ent}
h_t^{D^\pi\mathcal{F}(X)}(\Phi|_{D^\pi\mathcal{F}(X)}) = h_{-t}^{D^\pi\mathcal{W}(X)}(\Phi^{-1}) = \lim_{N \to \infty} \frac{1}{N} \log \sum_{k \in \mathbf{Z}} \dim \mathrm{Hom}^{-k}(L,\Phi^N(S)) e^{kt}
\end{equation}
by Theorem \ref{app-thm}.
In particular, if $L \in D^\pi\mathcal{W}(X)$ and $S \in D^\pi\mathcal{F}(X)$ are direct sums of some exact Lagrangian submanifolds and $\Phi$ is induced from an exact symplectomorphism $\phi : (X,\theta) \to (X,\theta)$ then the equation \eqref{fuk-ent} can be interpreted as the exponential growth rate of the geometric intersection number of $L$ and $\phi^N(S)$.

It seems that the notion of ST-triple is very useful when studying some properties of $D^\pi\mathcal{W}(X)$ relative to $D^\pi\mathcal{F}(X)$ or comparing $D^\pi\mathcal{W}(X)$ and $D^\pi\mathcal{F}(X)$.
Further interesting consequences of this perspective on Fukaya categories will be explored in a future work in progress \cite{BJK}.

\end{document}